\documentclass[12pt]{amsart}

\usepackage{amsmath, graphicx, cite, enumerate, comment}
\usepackage{amssymb, amsthm, amscd}
\usepackage{latexsym}
\usepackage{amssymb}
\usepackage[english]{babel}
\usepackage{amsmath,amssymb}
\usepackage{amsfonts}
\usepackage{cite}
\usepackage{marvosym}
\usepackage{mathrsfs}
\usepackage{amsthm}
\usepackage{hyperref}
\usepackage{verbatim}

\usepackage{tikz} \usetikzlibrary{arrows,shapes,patterns} 
\usepackage{lipsum} 

\newtheorem{thm}{Theorem}
\newtheorem{lem}[thm]{Lemma}
\newtheorem{cor}[thm]{Corollary}
\newtheorem{prop}[thm]{Proposition}

\theoremstyle{remark}
\newtheorem{rmk}[thm]{Remark}

\theoremstyle{definition}

\numberwithin{equation}{section}

\newcommand{\R}{\mathbb{R}}

\newcommand{\la}{\langle}
\newcommand{\ra}{\rangle}

\newcommand{\be}{\begin{equation}}
\newcommand{\ee}{\end{equation}}

\newcommand{\ben}{\begin{equation*}}
\newcommand{\een}{\end{equation*}}

\def\barroman#1{\sbox0{#1}\dimen0=\dimexpr\wd0+1pt\relax
  \makebox[\dimen0]{\rlap{\vrule width\dimen0 height 0.06ex depth 0.06ex}%
    \rlap{\vrule width\dimen0 height\dimexpr\ht0+0.03ex\relax 
            depth\dimexpr-\ht0+0.09ex\relax}%
    \kern.5pt#1\kern.5pt}}

\newcommand{\twoparteq}[4]
{
	\left\{
		\begin{array}{ll}
			#1 & \mbox{  } #2 \bigskip \\
			#3 & \mbox{  } #4
		\end{array}
	\right.
}

\def\XXint#1#2#3{{\setbox0=\hbox{$#1{#2#3}{\int}$}
     \vcenter{\hbox{$#2#3$}}\kern-.5\wd0}}

\newcommand{\pl}[2]{{\frac{\partial #1}{\partial #2}}}

\newcommand{\de}{\partial}

\newcommand{\N}{\mathbb{N}}
\newcommand{\h}{\mathscr{H}}

\newcommand{\emb}{\hookrightarrow}
\newcommand{\cemb}{\subset \subset}
\newcommand{\In}{\subset}
\newcommand{\Om}{\Omega}
\newcommand{\om}{\omega}
\newcommand{\dl}{{\delta}}
\newcommand{\Dl}{{\Delta}}

\newcommand{\al}{{\alpha}}

\newcommand{\id}{\ d}

\newcommand{\D}{{\nabla}}
\newcommand{\Db}{\nabla^{\bot}}
\newcommand{\ti}[1]{{\tilde{#1}}}
\newcommand{\eps}{{\varepsilon}}
\newcommand{\fr}[2]{\frac{#1}{#2}}
\newcommand{\sm}{{\setminus}}

\newcommand{\Nn}{\mathcal{N}}

\newcommand{\Dv}{{\rm{div}}}

\newcommand{\vlinesub}[1]{\vline_{_{_{_{_{_{_{_{#1}}}}}}}}}

\newcommand{\floor}[1]{\lfloor #1 \rfloor}

\def\XXint#1#2#3{{\setbox0=\hbox{$#1{#2#3}{\int}$}
     \vcenter{\hbox{$#2#3$}}\kern-.5\wd0}}

\begin{document} 	               
       \thanks{\textsl{Mathematics Subject Classification (MSC 2010):} Primary 53A10; Secondary 53C42, 49Q05.}
     \title[Compactness analysis for free boundary minimal hypersurfaces]{Compactness analysis for free boundary minimal hypersurfaces}
     \author{Lucas Ambrozio, Alessandro Carlotto and Ben Sharp}
     \address{\noindent  L. Ambrozio: University of Warwick, Gibbet Hill Rd, Coventry CV4 7AL, United Kingdom \textit{E-mail address: L.Coelho-Ambrozio@warwick.ac.uk} \newline \newline 
     \indent A. Carlotto: ETH - Department of Mathematics, R\"amistrasse 101, 8092 Z\"urich,Switzerland, \textit{E-mail address: alessandro.carlotto@math.ethz.ch} \newline \newline 
     \indent B. Sharp: University of Warwick, Gibbet Hill Rd, Coventry CV4 7AL, United Kingdom \textit{E-mail address: B.Sharp@warwick.ac.uk}}

     	\begin{abstract}We investigate compactness phenomena involving free boundary minimal hypersurfaces in Riemannian manifolds of dimension less than eight. We provide natural geometric conditions that ensure strong one-sheeted graphical subsequential convergence, discuss the limit behaviour when multi-sheeted convergence happens and derive various consequences in terms of finiteness and topological control.
     			\end{abstract}
     	
     	\maketitle

\section{Introduction}

Let $(\Nn^{n+1},g)$ be a compact Riemannian manifold with boundary, $n\geq 2$, and $M^n$ be a properly embedded codimension one submanifold: we say that $M^n$ is a free boundary minimal hypersurface in $(\Nn^{n+1},g)$ if it is a critical point for the $n$-dimensional area functional under the sole constraint that $\partial M\subset\partial\mathcal{N}$ or, equivalently, if $M$ has zero mean curvature and meets the ambient boundary orthogonally. \\
\indent The study of free boundary minimal surfaces, starting from the special setting of Euclidean domains, goes back at least to Courant \cite{Cou40, Cou50} and has attracted considerable attention for several decades. In recent years, various interesting research lines have emerged. First of all, the work of Fraser and Schoen \cite{FS11, FS13, FS16} has clarified the link of these geometric objects with extremal metrics, of given volume, for the first Steklov eigenvalue of manifolds with boundary. Secondly, one has witnessed a few interesting classification results towards conjectural characterisations of the critical catenoid among free boundary minimal hypersurfaces in the Euclidean unit ball \cite{AN16, Dev16, M16, SZ16, Tra16}. Thirdly, and perhaps most relevantly to this paper, various techniques have been developed to prove existence results and produce novel concrete examples. In the Euclidean ball, Fraser and Schoen have obtained free boundary minimal surfaces with genus zero and any number of boundary components larger than one \cite{FS16}, while Folha, Pacard and Zolotareva \cite{FPZ15} have obtained free boundary minimal surfaces with genus zero or one and any sufficiently large number of boundary components. Ketover \cite{Ket16A} proposed an equivariant min-max approach to generate free boundary minimal surfaces with given discrete symmetry group, thereby producing examples with arbitrarily large genus. In more general settings, we shall mention here the results by Li \cite{Li15}, Li-Zhou \cite{LZ16B}, De Lellis-Ramic \cite{DR16} and Maximo-Nunes-Smith \cite{MNS13}, the former three works extending the min-max theory to the free boundary setting and the latter being based on a degree-theoretic approach. Lastly, in higher dimensions, equivariant techniques in the spirit of Hsiang \cite{Hsi83} have been adapted to the free boundary context by Freidin, Gulian and McGrath \cite{FGM16}. \\
\indent The primary scope of this article is to investigate compactness phenomena involving free boundary minimal hypersurfaces, in particular to single out natural geometric conditions that imply strong one-sheeted graphical subsequential convergence, and describe various relevant phenomena when instead multi-sheeted convergence occurs. This analysis leads to interesting geometric conclusions concerning the space of free boundary minimal hypersurfaces inside certain classes of ambient Riemannian manifolds. \\
\indent The starting point of our discussion is the following result by Fraser and Li \cite{FL14}, which should be regarded as a free boundary analogue of the classic theorem by Choi and Schoen \cite{CS85}:

\begin{thm}[Cf. \cite{FL14} Theorem 1.2]\label{thm:3dfraserli}
Let $(\Nn^3,g)$ be a compact Riemannian manifold with non-empty boundary. Suppose that $\Nn^3$ has non-negative Ricci curvature and strictly convex boundary. Then the space of compact, properly embedded, free boundary minimal surfaces of fixed topological type in $\Nn^3$ is compact in the $C^k$ topology for any $k\geq 2$.
\end{thm}	

Of course, one may wonder whether a similar conclusion holds true in higher dimensions, namely whether control on the topology ensures subsequential convergence in the strong sense explained above. In fact, this is definitely not the case: by \cite{FGM16}, for any couple of integers $m, n\geq 2$ such that $m+n<8$ there exists an infinite family of distinct, free boundary minimal hypersurfaces in the Euclidean unit ball of dimension $m+n$, all having the topological type of $D^m\times S^{n-1}$ and converging to a singular limit. For instance, in the case $m=n=2$ the family they construct converges to a cone over a Clifford torus. Furthermore the `second principal family'  constructed by Hsiang in 1983 (see \cite{Hsi83}) provides infinite examples of free boundary minimal hypersurfaces all with the same topology (namely that of $D^2\times S^1$) inside the upper hemisphere $S^4_{+}$, but again the limit of these hypersurfaces is singular. Therefore, there are counterexamples both in the case when either the Ricci tensor vanishes on the interior and the boundary is strictly convex, or the Ricci tensor is positive on the interior and the boundary is weakly convex. \\

In order to proceed and state our results, we let
\ben
\mathfrak{M}:=\left\{ M\emb \Nn \ \left|
		\begin{array}{l}
			   \text{$M$ is a smooth, connected, compact, properly embedded free boundary}\\
			  \text{minimal hypersurface with respect to $\de \Nn$}
		\end{array}
	\right.\right\}. \een
\
\
Notice that the class $\mathfrak{M}$ is in general not closed under smooth graphical convergence with multiplicity one: easy examples of non-convex domains in $\R^3$ show that the limit of elements in $\mathfrak{M}$ may be not properly embedded, and in fact have a large contact set with the boundary of the ambient manifold. We shall introduce here the following general assumption:

\begin{center}
	$(\textbf{P})$ \ \ if $M\subset \mathcal{N}$  has zero mean curvature and meets the boundary of the ambient manifold orthogonally along its own boundary, then it is proper.
\end{center}

 For instance, it is readily seen via a standard application of the maximum principle that the assumption above is implied by this natural geometric requirement:
\begin{center}
	$(\textbf{C})$ \ \ $\partial \mathcal{N}$ is weakly mean convex and has no minimal component.	
\end{center}



\indent An element $M\in \mathfrak{M}$ can be either two-sided or one-sided, in other words the associated normal bundle can be either trivial or not respectively. Therefore, we further define  
$\widetilde{\mathfrak{M}}$ to be the class of immersed free boundary minimal hypersurfaces $\widetilde{M}$ that are the two-sided covering of some one-sided $M\in \mathfrak{M}$. One can check that $\widetilde{M}$ always exists given a one-sided $M\in \mathfrak{M}$, that these submanifolds are connected, and the covering map gives rise to a two-sided free boundary, minimal, proper immersion of $\widetilde{M}$ into $\Nn$. The construction of $\widetilde{M}$ is analogous (and equivalent in the case that $\Nn$ is orientable) to the construction of the orientable double cover of a non-orientable manifold. 

Within the class $\mathfrak{M}$, we denote by $\lambda_p \in \R$ the value of the $p$-th eigenvalue of the stability operator of a given element $M\in \mathfrak{M}$ (see Subsection \ref{subs:spectrum}) and consider for given $\mu, \Lambda  \in \R_{\geq 0}$, $p\in \N$ the subset
\[
\mathfrak{M}_p(\Lambda, \mu) :=\{M\in \mathfrak{M}\ | \ \text{$\lambda_p(M)\geq -\mu$ and $\h^n(M) \leq \Lambda$} \}.
\]
We let $index(M) = \max\{k\ | \ \lambda_k <0\}$ and $nullity(M)$ be the dimension of the eigenspace associated to the zero eigenvalue. When $M\in \mathfrak{M}$ has index (resp. nullity) zero we shall say it is a stable (resp. non-degenerate), free boundary minimal hypersurface.

Our first result asserts that a uniform lower bound on \textsl{some} eigenvalue of the Jacobi operator together with a uniform upper bound on the area is sufficient for a weak compactness result, in the sense of graphical but possibly multi-sheeted convergence away from finitely many points where necks (or half-necks, at boundary points) may form.

\begin{thm}\label{thm_weak_comp}
Let $2\leq n\leq 6$ and $(\Nn^{n+1},g)$ be a compact Riemannian manifold with boundary satisfying the assumption $(\textbf{P})$. For fixed $\Lambda, \mu\in \R_{\geq 0}$ and $p\in \N$, suppose that $\{M_k\}$ is a sequence in $\mathfrak{M}_p(\Lambda, \mu)$. Then there exist $M\in \mathfrak{M}_p(\Lambda, \mu)$, $m\in \mathbb{N}$ and a finite set $\mathcal{Y}\In M$ with cardinality $|\mathcal{Y}|\leq p-1$  such that, up to subsequence, $M_k \to M$ locally smoothly and graphically on $M\sm \mathcal{Y}$ with multiplicity $m$.  
\end{thm}

\begin{rmk}\label{rem:massbound}
The multiplicity $m$ satisfies $m\leq \frac{\Lambda}{\eps(\Nn)}$ where $\eps(\Nn) =\inf\{\h^n(M)|M\in \mathfrak{M}\}$. The constant $\eps(\Nn)$ is positive since $(\Nn,g)$ is compact and a monotonicity formula holds, see Subsection \ref{subs:curvest}.  
\end{rmk}

\begin{rmk}\label{rem:extension}
If assumption $(\textbf{P})$ is dropped, it is still possible to show that the sequence $\left\{M_k\right\}$ subconverges to a limit hypersurface, which is smooth, minimal and meets the boundary of the ambient manifold orthogonally along its own boundary, but which may fail in general to be proper, as it may have an interior contact set with $\partial\mathcal{N}$, as discussed above. We refer the reader to Section \ref{sec:thm2} for a broader discussion and for the proof of the more general Theorem \ref{thm_weak_comp_ext}, from which Theorem \ref{thm_weak_comp} follows at once.  	
\end{rmk}

A version of the theorem below appeared in \cite{Sha15}, assuming a uniform upper bound on the Morse index, and later in \cite{ACS15} assuming a lower bound on the $p$-the eigenvalue of the Jacobi operator, for what concerns the case of \textsl{closed} minimal hypersurfaces (see also \cite{A16}, Proposition 4.1). Here we treat the general case, in arbitrary dimension $n\geq 2$, describing what happens when a properly embedded minimal hypersurface is the limit  of free boundary minimal hypersurfaces in the sense stated in Theorem \ref{thm_weak_comp}. In order to state the remaining results concisely we introduce some notation; given two smooth manifolds $M_1, M_2$ (possibly with non-empty boundary), we shall write $M_1\simeq M_2$ if they are diffeomorphic, and for a subset $\mathfrak{C}\subset\mathfrak{M}$ we let $\mathfrak{C}/\simeq $ denote the set of corresponding equivalence classes modulo diffeomorphisms.

\begin{thm}\label{thm_mult_anal}
Let $n\geq 2$ and $(\Nn^{n+1}, g)$ be a compact Riemannian manifold with boundary. Suppose that $\{M_k\}$ is a sequence in $\mathfrak{M}$ such that there exists some $M\in \mathfrak{M}$ and a finite set $\mathcal{Y}\In M$ with $M_k\to M$ locally smoothly and graphically on $M\sm \mathcal{Y}$ with multiplicity $m \in \N$. If $M_k\neq M$ eventually, then either $nullity(M)\geq 1$ or if $M$ is one-sided, $nullity(\widetilde{M})\geq 1$ where $\widetilde{M}$ is the two-sided immersion associated to $M$. Furthermore:
\begin{enumerate}
\item If $M$ is two-sided 
\begin{enumerate}[i)]
\item $m=1$ if and only if $\mathcal{Y}=\emptyset$, and $M_k\simeq M$ eventually
\item $m\geq 2$ if and only if $\mathcal{Y}\neq\emptyset$, and $M$ is stable with $nullity(M) =1$. 
\end{enumerate}
\item If $M$ is one-sided 
 \begin{enumerate}[i)]
  \item $m=1$ implies $\mathcal{Y}=\emptyset$ and $M_k\simeq M$ eventually
  \item $m\geq2$ implies $\widetilde{M}$ is stable, $nullity(\widetilde{M}) =1$ and $\lambda_1(M) >0$. In this case $\mathcal{Y}=\emptyset$ implies $m=2$ and $M_k\simeq \widetilde{M}$ eventually.  
\end{enumerate}
\end{enumerate}
\end{thm}

\begin{rmk} 
Notice that when one allows limit hypersurfaces that are not properly embedded, it is not possible to rule out the case $m=1$ and $\mathcal{Y}\neq \emptyset$. The local picture of what goes wrong is provided by half a catenoid in $\R^{n+1}$, that is free boundary with respect to $\{x^{n+1}=0\}$ (the boundary in this case being an $n-1$-dimensional sphere), which is blown down to converge with multiplicity one to the plane $\{x^{n+1}=0\}$, with the set $\mathcal{Y}$ consisting of one point, namely the origin.
\end{rmk}

A first consequence, of wide geometric applicability, is the following.

\begin{cor}\label{cor_strong_comp1}
Let $2\leq n\leq 6$ and $(\Nn^{n+1},g)$ a compact Riemannian manifold with boundary satisfying the assumption $(\textbf{P})$. Suppose that every  properly embedded free boundary minimal hypersurface is unstable. Then the corresponding class $\mathfrak{M}_p(\Lambda, \mu)$ is sequentially compact in the smooth topology of single-sheeted graphical convergence and thus the quotient $\mathfrak{M}_p(\Lambda, \mu)/\simeq$ only consists of finitely many equivalence classes.
In particular, such conclusion holds when $\Nn$ satisfies either 
\begin{enumerate}[i)]
	\item $Ric_\mathcal{N} \geq 0$ with $\de \Nn$ strictly convex, or
	\item $Ric_\mathcal{N} >0$ with $\de \Nn$ weakly convex and strictly mean convex.  
\end{enumerate}
\end{cor}

In particular, when $n=2$ and $(\mathcal{N}^3,g)$ satisfies assumption \textsl{i)} this result (together with Theorem \ref{thm:3dfraserli}) implies the equivalence for a subclass $\mathfrak{M}'$ of $\mathfrak{M}$ of the following two statements:
\begin{enumerate}
\item[a)]{$\mathfrak{M}'/\simeq$ is finite, namely $\mathfrak{M}'$ contains elements belonging to finitely many diffeomorphism classes;}
\item[b)]{for any integer $p\geq 1$ there exist constants $\Lambda_p, \mu_p\in\mathbb{R}_{\geq 0}$ (both depending on $p$) such that $\mathfrak{M}'\subset\mathfrak{M}_{p}(\Lambda_p,\mu_p)$.}
\end{enumerate}	 
It follows that, in manifolds of non-negative Ricci curvature and strictly convex boundary, one can regard Corollary \ref{cor_strong_comp1} as a strong compactness theorem which suitably extends Theorem \ref{thm:3dfraserli} up to ambient dimension seven. \\

\indent Furthermore, we can derive a finiteness result when the ambient manifold is known not to admit degenerate minimal hypersurfaces:

\begin{cor}\label{cor_generic}
Let $2\leq n\leq 6$ and $(\Nn^{n+1},g)$ be a compact Riemannian manifold with boundary satisfying the assumption $(\textbf{P})$.  Suppose that for all $M\in \mathfrak{M}$ and $\widetilde{M}\in \widetilde{\mathfrak{M}}$ there exist no non-trivial Jacobi fields over $M$ or $\widetilde{M}$. Then $\mathfrak{M}_p(\Lambda, \mu)$ contains finitely many elements.  
\end{cor}

Of course, one may wonder whether the condition we need to assume in the statement above, namely the condition 

\begin{center}
\textsl{for all $M\in \mathfrak{M}$ and $\widetilde{M}\in \widetilde{\mathfrak{M}}$ there exist no non-trivial Jacobi fields over $M$ or $\widetilde{M}$}
\end{center}
is generically satisfied with respect to the background metric we endow $\mathcal{N}$ with. This is a question of independent interest, which we answer by proving a suitable `bumpy metric theorem' in the category under consideration. In the statement below, we let $\Gamma^{q}$ be the set of $C^{q}$ metrics on $\mathcal{N}$ endowed with the $C^{q}$ topology.

 \begin{thm}\label{thm:bumpy} 
 	Let $\mathcal{N}^{n+1}$ be a smooth, compact, connected manifold with non-empty boundary, and $q$ denote a positive integer $\geq 3$, or $\infty$. \\
 	\indent Let $\mathcal{B}^{q}$ be the subset of metrics $g$ in $\Gamma^{q}$ defined by the following property: no compact smooth manifolds with boundary that are $C^{q}$ properly embedded as free boundary minimal hypersurfaces in $(\mathcal{N},g)$, and no finite covers of any such hypersurface, admit a non-trivial Jacobi field. Then $\mathcal{B}^q$ is a comeagre subset of $\Gamma^{q}$. 
 \end{thm}

This is the free boundary counterpart of Theorem 2.2 in \cite{Whi91} (for finite $q$) and Theorem 2.1 in \cite{Whi15} (for $q=\infty$). 
The definition of comeagre set and an appropriate contextualization of this result are provided in the first part of Section \ref{sec:bumpy}, see in particular Remark \ref{rmk:Baire}. One can derive a direct geometric application by simply combining Theorem \ref{thm:bumpy} with Corollary \ref{cor_generic}.
 
 \begin{cor}\label{cor:mcbumpy}
Let $2\leq n\leq 6$ and $\Nn^{n+1}$ be a compact smooth manifold with boundary. For a generic choice $g$ in the class of Riemannian metrics such that the boundary $\partial\mathcal{N}$ is strictly mean convex, the subclass $\mathfrak{M}_p(\Lambda,\mu)$ of free boundary minimal hypersurfaces in $(\mathcal{N}^{n+1},g)$ only contains finitely many elements.
 \end{cor}	 

We now discuss some variations of the theorems we have presented above.

\begin{rmk}
Since, for $I\in \N$ we have
$$ \mathfrak{M}(\Lambda, I) := \mathfrak{M}_{I+1}(\Lambda,0)=\{ M\in \mathfrak{M} \ | \ \text{$index(M) \leq I$ and $\h^n(M) \leq \Lambda$} \}$$
then Theorem \ref{thm_weak_comp} (in fact, Theorem \ref{thm_weak_comp_ext}), and Corollaries \ref{cor_strong_comp1}, \ref{cor_generic} and \ref{cor:mcbumpy} all hold for the class $\mathfrak{M}(\Lambda, I)$ (for fixed $\Lambda \in \R_{\geq 0}$, $I\in \N$). 
\end{rmk}

A result of Brian White \cite[Theorem 2.1]{Whi09} states that if $\partial \Nn$ is mean convex then: $\Nn$ contains no closed smooth and embedded minimal hypersurface if and only if any smooth hypersurface in $\Nn$ with boundary satisfies an isoperimetric inequality (with a uniform constant), namely
there exists some $C=C(\Nn)$ such that for all free boundary minimal hypersurfaces in $\Nn$ we have 
\[
\h^n(M)\leq C\h^{n-1}(\de M).
\]
That is certainly the case for compact mean convex subdomains in $\R^{n+1}$ but there are of course many other interesting examples. Motivated by this result, we define 
\[\mathfrak{M}^\de_p(\Lambda, \mu) :=\{M\in \mathfrak{M} \ | \ \text{$\lambda_p(M)\geq -\mu$ and $\h^{n-1}(\de M) \leq \Lambda$} \}.
\] 
and similarly 
\[  \mathfrak{M}^{\de}(\Lambda, I) :=\mathfrak{M}^\de_{I+1}(\Lambda, 0)=\{ M\in \mathfrak{M} \ | \ \text{$index(M) \leq I$ and $\h^{n-1}(\partial M)\leq \Lambda$} \}.
\]

\begin{cor}
Let $2\leq n\leq 6$ and $(\Nn^{n+1},g)$ be a compact Riemannian manifold with boundary and assume it has mean convex boundary and contains no closed minimal hypersurface. Then Theorem \ref{thm_weak_comp} (in fact, Theorem \ref{thm_weak_comp_ext}) and Corollaries \ref{cor_strong_comp1}, \ref{cor_generic} and \ref{cor:mcbumpy} hold both for the class $\mathfrak{M}^\de_p(\Lambda, \mu)$ and $\mathfrak{M}^{\de}(\Lambda, I)$. 
\end{cor}

In certain positively curved ambient three-manifolds, both the area bound and the bound on the length of the boundary can be dropped and a compactness result can be achieved by only assuming an upper bound on the Morse index. For instance, that is certainly the case when $(\mathcal{N},g)$ is a strictly convex domain in the Euclidean space, as a result of the combined application of the effective estimate proven in \cite{ACS16} and the aforementioned theorem by Fraser and Li, see Corollary E in \cite{ACS16} for a precise statement.\

\

\indent Let us conclude this introduction by describing the structure of the paper. After presenting our setup, we recollect in Section \ref{sec:setup} a number of basic facts concerning free boundary minimal hypersurfaces, with special emphasis on monotonicity formulae and curvature estimates. Section \ref{app:fol} and \ref{sec:remov_sing} are instead devoted to ancillary results of independent interest and applicability, namely the construction of a local minimal foliation around a boundary point of a free boundary minimal hypersurface and a boundary removable singularity theorem. The proofs of the two main theorems are respectively given in Section \ref{sec:thm2} for what concerns Theorem \ref{thm_weak_comp} and in Section \ref{sec:thm4} for Theorem \ref{thm_mult_anal}. Lastly, a thorough discussion of the bumpy metric theorem is presented in Section \ref{sec:bumpy}. The main body of the paper is complemented by two appendices that are devoted to certain technical aspects emerging in the proofs of some of our results. In particular, in Appendix \ref{app:2ndvar} we present a detailed derivation of the second variation formula for smooth hypersurfaces with boundary (without assuming either minimality of the hypersurface or orthogonal intersection with the ambient manifold). To the best of our knowledge, this computation is not easily found in the literature and believe it might be useful to present it here.

\

\textsl{Acknowledgments}. The authors wish to express their gratitude to Andr\'e Neves for his interest in this work and for his constant support, and to the anonymous referee for carefully reading the manuscript and providing detailed feedback.
A. C. also would like to thank Connor Mooney for several discussions and Francesco Lin for pointing out some relevant references. 
L. A. was visiting the University of Chicago while this article was written, and he would like to thank the Department of Mathematics for its hospitality.  He is supported by the EPSRC on a Programme Grant entitled `Singularities of Geometric Partial Differential Equations' reference number EP/K00865X/1.

\section{Setup and preliminaries}\label{sec:setup}

Let $(\mathcal{N}^{n+1},g)$ be a compact Riemannian manifold with boundary and $n\geq 2$. We let $\nabla$ denote the corresponding Levi-Civita connection, and $\nabla^{M}$ the induced connection on a submanifold $M^k$ for $k\leq n$ (in fact we shall only consider $k=n$ in our discussion). In the next three subsections we review the first and second variation formulae for free boundary minimal hypersurfaces and recall the basic definitions concerning the Jacobi operator and its spectrum.

\subsection{Free boundary minimal hypersurfaces}\label{subs:statvar}

Given the ambient manifold $(\mathcal{N},g)$,  we use the symbol $\mathfrak{X}_\de=\mathfrak{X}_\de(\mathcal{N})$ to denote the linear space of smooth ambient vector fields $X$ such that $a)$ $X(x)\in T_x\mathcal{N}$ for all $x\in\mathcal{N}$ and $b)$ $X(x)\in T_{x}\partial\mathcal{N}$ for all $x\in\partial\mathcal{N}$. To any such vector field  we can associate a one-parameter family of diffeomorphisms $\psi_t:\Nn\to\Nn$ for say $t\in(-\delta,\delta)$, where $\delta>0$. Notice that condition $b)$ allows points in $\partial\mathcal{N}$ to move under the flow $\psi$, but ensures $\psi_t(\partial\mathcal{N})\subset\partial\mathcal{N}$ for all $t$.
The first variation formula for varifolds (Cf. \cite{Sim83}, Section 39.2), specified to the case when $V$ is the varifold associated to a smooth, properly embedded, hypersurface $M\subset\mathcal{N}$ (by which we mean that $M\cap\partial\mathcal{N}=\partial M$), takes the form
\begin{equation}
\left.\frac{d}{dt}\h^n(M_t) \right\vert_{t=0}=\int_M \Dv_M (X)\id \h^n = -\int_M g(H,X) \id\h^n + \int_{\partial M} g(X,\nu) \id \h^{n-1}.
\end{equation}
Here we have set $M_t=\psi_t(M)$ and have denoted by $\nu$ the outward unit co-normal to $\partial M$.

Thus, it follows that $M$ is a stationary point for the area functional if and only if it has vanishing mean curvature ($H=0$ identically on $M$) and meets the ambient boundary $\partial\mathcal{N}$ orthogonally ($\nu\perp T\partial\mathcal{N}$ at all points of $\partial M$). In this case we shall say that $M$ is a \textsl{free boundary minimal hypersurface} in $(\mathcal{N},g)$.

\begin{rmk}
Along the course of our proofs, we will also have to deal with hypersurfaces $M$ whose boundary consists of two parts $\Sigma_1, \Sigma_2$ such that i) $\Sigma_1\cap \Sigma_2$ is a smooth, possibly disconnected, $(n-2)$-dimensional manifold ii) $\Sigma_1\subset \partial\mathcal{N}$ and iii) $\Sigma_2\setminus\Sigma_1\subset \mathcal{N}\setminus\partial\mathcal{N}$. With slight abuse of language, we say that $M$ is a minimal hypersurface with free boundary on $\partial\mathcal{N}$ if $M$ is minimal (namely has zero mean curvature at all points) and meets $\partial\mathcal{N}$ orthogonally along $\Sigma_1$.	
\end{rmk}	

\subsection{Second variation.}\label{subs:secvar}

Given a properly embedded, free boundary minimal hypersurface, we can then consider the second variation of the area functional. In the smooth setting under consideration, one has

\begin{eqnarray}
\left.\frac{d^2}{dt^2} \h^n (M_t) \right\vert_{t=0} 
&=& \int_M (|\D^\bot X^\bot|^2 - (Ric_\Nn(X^\bot,X^\bot)+|A|^2|X^\bot|^2)) \id \h^n \nonumber\\
&&+ \int_{\partial M} {\rm{II}}(X^\bot,X^\bot) \id \h^{n-1}
\end{eqnarray}
where:
\begin{itemize}
\item{for given $X\in\mathfrak{X}_\de$, the symbol $X^{\perp}$ denotes the normal component of the vector field in question, with respect to the tangent space of $M\subset\mathcal{N}$;}
\item{$\nabla^{\perp}$ is the induced connection on the normal bundle of $M\subset\Nn$;}
\item{$A$ is the second fundamental form of $M\subset\Nn$;}
\item{${\rm{II}}(W,Z)=-g(W,D_Z \hat{\nu})$ is the second fundamental form of $\partial\mathcal{N}\subset\mathcal{N}$ with respect to the outward unit normal  $\hat{\nu}$. Notice that with this convention one has that ${\rm{II}}(X^\bot,X^\bot) < 0$ whenever $\de \Nn$ is strictly convex and $X^\bot\neq 0$.} 
\end{itemize}	

\subsection{The Jacobi operator and its spectrum}\label{subs:spectrum}

Motivated by the explicit expression of the second variation formula, we consider for $M$ a properly embedded free boundary hypersurface and $v\in\Gamma(NM)$ (i. e. a section of the corresponding normal bundle) the quadratic form
 \begin{equation}\label{eq:Jacform}
 \check{Q}^M(v,v) := \int_{M} \left(|\nabla^{\perp}v|^2 - (Ric_{\Nn}(v,v) + |A|^2|v|^2)\right)\id\h^{n} + \int_{\partial M}{\rm{II}}(v,v)\id \h^{n-1}
 \end{equation}
that is often called index form of the free boundary minimal hypersurface in question. The index of $M$ is defined as the index of $\check{Q}^{M}$, that is, the maximal dimension of a linear subspace $E$ in $\Gamma(NM)$ such that $\check{Q}^M(v,v) < 0$ for all $v$ in $E\setminus\{0\}$. The index can be computed analytically in terms of the spectrum of a second order differential operator with oblique boundary conditions. Indeed, integration by parts gives
\begin{equation*}
\check{Q}_M(v,v) = - \int_{M} g(v,\check{\mathcal{L}}_{M}(v)) \id\h^n + \int_{\partial M} \left( g(v,\Db_{\nu}v) + {\rm{II}}(v,v)\right) \id\h^{n-1}
\end{equation*}
 where $\check{\mathcal{L}}_{M}v := \Delta^{\perp}_{M}v + Ric^{\perp}_{\mathcal{N}}(v,\cdot) + |A|^2 v$ is the Jacobi operator of $M$ (and we regard $Ric_{\mathcal{N}}$ as a $(1,1)$-tensor). The boundary condition
 \begin{equation*}
 g(\Db_{\nu}v,\cdot) = -{\rm{II}}(v,\cdot) 
 \end{equation*}
 is an elliptic boundary condition for $\check{\mathcal{L}}_{M}$, therefore there exists a non-decreasing and diverging sequence $\lambda_{1}\leq \lambda_{2}\leq\ldots \leq \lambda_{p} \nearrow \infty$ of eigenvalues  associated to a $L^{2}$-orthonormal basis $\{v_p\}_{p=1}^{\infty}$ of solutions to the eigenvalue problem
 \begin{equation*}
 \begin{cases}
 \check{\mathcal{L}}_{M}(v) + \lambda v  = 0 \quad & \text{on} \quad M, \\
  \Db_{\nu}v = -({\rm{II}}(v,\cdot))^{\sharp} \quad & \text{on} \quad \partial M.
 \end{cases} 
 \ \ \ (*)
 \end{equation*}
 In the formula above and throughout the paper we let the symbols $\sharp$ and $\flat$ denote the standard musical isomorphisms with respect to the background metric under consideration. \\ 
 \indent The index of the free boundary minimal hypersurface $M$ is then equal to the number of negative eigenvalues of the system $(*)$ above.
 The solutions of $(*)$ have a standard variational characterisation: If $E_{p}$ denotes a $p$-dimensional subspace of $\Gamma(NM)$ then
  \begin{equation*}
 \lambda_{p}(\check{\mathcal{L}}_M) = \inf_{E_p}\max_{v \in E_{p}\sm\{0\}} \frac{\check{Q}^{M}(v,v)}{\int_{M}|v|^2\id \h^{n}}.
 \end{equation*}
 \indent The min-max value is attained precisely by eigenfunctions of $\check{\mathcal{L}}_{M}$ associated to $\lambda_{p}$ and satisfying the boundary conditions in $(*)$. \\
 \indent In the case that $M$ is two-sided we can (by picking a global unit normal $N$ over $M$) identify sections $v\in \Gamma(NM)$ with smooth functions $f\in C^{\infty}(M)$, so that one has in fact $\check{Q}_M(v,v)=Q_M(f,f)$ for $v=fN$ and
 \begin{equation*}
Q_{M}(f,f) = - \int_{M} f\mathcal{L}_{M}(f) \id\h^n + \int_{\partial M}  \left(f\frac{\partial f}{\partial\nu} + f^2{\rm{II}}(N,N)\right) \id\h^{n-1}
\end{equation*}
 having set $\mathcal{L}_{M} := \Delta_{M} + Ric_{\mathcal{N}}(N,N) + |A|^2 $, the scalar Jacobi operator of $M$. Standard arguments allow to conclude that the first eigenvalue for $(*)$ is simple (i. e. $\lambda_1 < \lambda_2$) and a corresponding first eigenfunction $f_1$ can be chosen to be strictly positive on $M$. As a result, if there exists a strictly positive Jacobi function, namely a non-trivial positive solution to $\mathcal{L}_M f = 0$ with $\frac{\partial f}{\partial \nu}+II(N,N)f=0$ on $\partial M$ , then $\lambda_1 =0$ ($M$ is stable) and $nullity(M)=1$. 
 
 In the sequel we will often deal with the local eigenvalues (i.e. local free boundary variations): given an open set $U\In \Nn$ we write $\lambda_p (M\cap U)$ to mean the $p^{th}$ eigenvalue with respect to zero boundary conditions on $\de (M\cap U) \sm \de M$. In other words we find $\lambda_{p}$ as above through sections $v$ that have relatively compact support in $M\cap U$, so that the associated solution to the eigenvalue problem is 
 \begin{equation*}
 \begin{cases}
 \check{\mathcal{L}}_{M}(v) + \lambda v  = 0 \quad & \text{on} \quad M\cap U, \\
  \Db_{\nu}v = -({\rm{II}}(v,\cdot))^{\sharp} \quad & \text{on} \quad \partial M\cap U \\
  v = 0 \quad &\text{on} \quad \de (M\cap U)\sm \de M.
 \end{cases} 
 \end{equation*}
 Clearly the usual monotonicity property holds: if $U_1\subset U_2\subset\mathcal{N}$ are open sets then the variational characterization above implies that $\lambda_p(M\cap U_1)\geq \lambda_p(M\cap U_2)$ for any positive integer $p$, with strict inequality whenever the inclusion is proper.

 \subsection{Monotonicity and curvature estimates}\label{subs:curvest}

In this subsection, we first present a suitable monotonicity formula, which is known to hold at boundary points of a free boundary minimal hypersurface. 
\begin{thm}[Cf. \cite{LZ16A} Theorem 3.5, see also \cite{GJ86B}]\label{thm_mon}
Let $M$ be a free boundary minimal hypersurface in $(\mathcal{N},g)$. For all $p\in \de \Nn$ there exist some $\Gamma = \Gamma(\Nn) >0$ and $0<r_0 = r_0(\Nn)$ such that for all $0<\sigma<\rho \leq r_0$
\begin{eqnarray*}
e^{\Gamma\sigma}\frac{\h^n(M\cap B_\sigma(p))}{\sigma^n}\leq e^{\Gamma\rho}\frac{\h^n(M\cap B_\rho(p))}{\rho^n} - F(M,\rho,\sigma,\Gamma)
\end{eqnarray*}
for a non-negative function $F(M,\rho,\sigma,\Gamma)$. Here $\Gamma$ depends on the second fundamental form of the embedding $\Nn\emb \R^d$ and on the second fundamental form of $\de \Nn \emb \Nn$. 
\end{thm}
\begin{rmk}\label{rmk_mon}
For fixed $\sigma$, $F$ is non-decreasing in $\rho$ and furthermore $F\equiv 0$ for $r\in [\sigma,\rho]$ if and only if $M$ is orthogonal to $\de B_r(p)$ for all $r$ in this range - i.e. when $M$ coincides with a cone in $B_\rho (p)\sm B_\sigma (p)$.   

In the case of the upper half space $\Nn= \R^{n+1}_+$ with $\de \Nn = \{x_{n+1} = 0\}$ we have $\Gamma \equiv 0$. In particular, and still in this case, if $\frac{\h^n(M\cap B_r(p))}{\om_n r^n} = \fr12$ for some $p\in \de M$ and $r>0$ then $M$ is a half plane, orthogonal to $\{x_{n+1}=0\}$ and passing through $p\in \de \R^{n+1}_+$. This follows easily from the fact that $\lim_{r\to 0}\frac{\h^n(M\cap B_r(p))}{\om_n r^n} = \fr12$ and therefore by monotonicity we must have that $M$ is a cone - but the only smooth free boundary minimal cone is a half plane in this case.
\end{rmk}

Coupled with the (usual) interior monotonicity formula, we state a simple corollary of the above when we restrict to $M\in \mathfrak{M}$. 
\begin{cor}\label{cor_mon}
Suppose $M\in \mathfrak{M}$ with $\h^n(M)\leq \Lambda$. There exists some $0<r_0=r_0(\Nn)$ and $C=C(\Nn, \Lambda)$ such that for all $p\in M$ and $0< r<r_0$, we have 
\begin{equation*}
\frac{1}{4} \leq \frac{\h^n(M\cap B_r(p))}{\om_n r^n} \leq C.
\end{equation*}
\end{cor}
In particular, this easily tells us that there exists $\eps = \eps(\Nn)>0$ so that, as was stated in Remark \ref{rem:massbound}, $\inf\{\h^n(M)|M\in \mathfrak{M}\} \geq \eps.$

\

We now recall a suitable version of Allard's regularity theorem in the smooth setting, which is inspired by a result due to Brian White \cite[Theorem 1.1]{Whi05} in the case of interior points. The general statement below follows easily via an analogous scaling argument, but relying on Remark \ref{rmk_mon} as well.
\begin{thm}[Cf. \cite{Whi05} Theorem 1.1]\label{thm_bound_allard}
Suppose that $M\in \mathfrak{M}$. For any $r_1\leq r_0$ (with $r_0$ as in the statement of Theorem \ref{thm_mon}) there exist $\eps=\eps(\Nn, r_1)$ and $C=C(\Nn, r_1)$ so that if  
$$\text{$\frac{\h^n(M\cap B_r(p))}{\om_n r^n} \leq (1+\eps)$ for all $p\in M\sm \de M$, $r<\min\{r_1,dist(p,\de M)\}$}$$ and 
$$\text{$\frac{\h^n(M\cap B_r(p))}{\om_n r^n} \leq \frac12 (1+\eps)$ for all $p\in \de M$, $r<r_1$}$$
then $\sup_M |A| \leq C$. 
\end{thm}

We recall that $L^\infty$ control on the second fundamental form $|A|$ and the area of a minimal hypersurface is enough to yield a strong compactness result. This is well-known and extends easily to the case of free boundary minimal hypersurfaces. 

\begin{thm}[Cf. \cite{LZ16A} Theorem 6.1]\label{thm_local_comp}
Let $\{M_k^n\}\In \mathfrak{M}$ and assume that
$$\h^n(M_k\cap B_{\rho}(p))+\sup_{M_k\cap B_{\rho}(p)} |A_k|\leq \Lambda<\infty$$ and there exists a sequence $p_k\in M_k$ with $p_k\to p$. Then there exists some $M\emb B_\rho(p)$ a smooth, embedded, minimal hypersurface, free boundary with respect to $\de \Nn$ and such that (up to subsequence) $M_k\to M$ uniformly smoothly and graphically on $B_{\frac{\rho}{2}}(p)$. 
\end{thm}

Finally, the curvature estimates of Schoen and Simon \cite{SS81} (or Schoen-Simon-Yau when $n\leq 5$, see \cite{SSY75}) hold up to the boundary for stable and free boundary minimal hypersurfaces. 
\begin{thm}[Cf. \cite{LZ16A} Theorem 1.1]\label{thm_A_bound}
Let $M\in \mathfrak{M}$ for $2\leq n\leq 6$, $p\in M$ and suppose that $\lambda_1(M\cap B_{\rho}(p))\geq -\mu$ for some $0\leq \mu <\infty$. If $\rho^{-n} \h^n (M\cap B_{\rho}(p)) \leq \Lambda$ then for all $x\in B_{\frac{\rho}{2}} (p)$
$$ |A|(x)\leq \fr{C}{dist_\Nn(x,\partial B_{\rho}(p))}$$
for some $C=C(n,\mu,\Nn,\Lambda)<\infty.$ 
\end{thm}
\begin{rmk}\label{rmk_A_bound}
The above estimate follows from the curvature estimates of Schoen-Simon \cite{SS81} which imply that the only complete, properly embedded, free boundary and stable minimal hypersurface in the upper half space $\R^{n+1}_+$ which has Euclidean volume growth is a half plane. 

This theorem was proved under slightly different hypotheses in \cite{LZ16A}, in particular they consider immersed and stable $M$, at which point the conclusion is only known to hold for $n\leq 5$ using the estimates of Schoen-Simon-Yau \cite{SSY75}. However, under the assumption of embeddedness of $M$, and a lower bound on the first eigenvalue, the conclusion holds also for $2\leq n\leq 6$ by the curvature estimates of Schoen-Simon (see page 13 of \cite{LZ16A}). Of course if $p\in M\sm \de M$ and $\rho<dist(p,\de \Nn)$ then this is just the (usual) interior curvature estimate of Schoen-Simon \cite[Corollary 1]{SS81}.

\end{rmk}

\section{Local minimal foliations around a boundary point of a free boundary minimal hypersurface}\label{app:fol}
   
   At various points in the proofs of our main results we will require the existence of a local foliation by minimal half-discs that are free boundary with respect to $\de\Nn$. In particular this construction will be crucial in proving the singularity removal theorem in Section \ref{sec:remov_sing}. The interior version, i. e. a local foliation by minimal discs, can be found in the Appendix of \cite{Whi87B}.          	
           
             Given $n\geq 2$ and an angle $\theta\in \left(0,\pi/4\right)$ let $p$ be the point on the negative $x^1$-axis in $\R^n$ such that the coordinate unit sphere $\partial B_1(p)$ meets the hyperplane $\left\{x^1=0\right\}$ at an angle $\theta$.  Then, let us consider the domain
             \[
             S_{\theta}:=\left\{x\in B_1(p) \ | \ x^1>0 \right\}
             \]	
             and the associated cylinder
             \[
             C_{\theta}=S_{\theta}\times\R=\left\{(x,x^{n+1}) \ | \ x\in S_{\theta}, \ x^{n+1}\in\R\right\}.
             \]	
             Set
             \[
             \Gamma_1=\overline{B_1(p)}\cap\left\{x^1=0\right\}, \ \ \Gamma_2=\partial \overline{S_{\theta}}\cap\left\{x^1\geq 0\right\}.
             \]	
             
             Given $\alpha_0\in (0,1)$ to be specified later (see Appendix \ref{app_fol}), so that suitable regularity results for elliptic problems with mixed boundary conditions are applicable, let $\alpha\in (0,\alpha_0)$ and consider the following H\"older functional spaces:
             \begin{equation*}
                X:= \left\{\textrm{Riemannian metrics of class} \ C^{2,\alpha} \ \textrm{on} \ \overline{C_{\theta}} \right\}, \ \	
             Y:= C^{2,\alpha}(\overline{S_{\theta}}) 
             \end{equation*}
             \begin{equation*}
             Z_1:=C^{0,\alpha}(\overline{S_{\theta}}), \ \ 
             Z_2:=C^{1,\alpha}(\Gamma_1), \ \
             Z_3:=C^{2,\alpha}(\Gamma_2). 
               \end{equation*}	
              
             We define the map
             \[
             \Phi: \R\times X\times Y\times Y \to \ Z_1\times Z_2\times Z_3
                \]
                by
                \[
                \Phi(t,g,w,u)= (H_g(t+w+u), g(N_g(t+w+u), \nu_g(t+w+u)), u|_{\Gamma_2})
                \]  
                where
                \begin{eqnarray*}
                H^N_g(t+w+u)&=& \textrm{mean curvature, in metric $g$, of the graph of the function} \\
                &&\ t+w+u: \overline{S_{\theta}}\to\R, \\
                N_g(t+w+u)&=& \textrm{upward unit normal, in metric $g$, of the graph of the function} \\
                &&\ t+w+u: \overline{S_{\theta}}\to\R, \\
                \nu_g(t+w+u)&=&\textrm{outward-pointing unit normal, in metric $g$, to} \ \Gamma_1\times\R, \ \textrm{evaluated at} \\
                &&\ t+w(x)+u(x).
                \end{eqnarray*}
                
                In order to avoid ambiguities, let us specify that in the first row we mean $H^N_g=g(H,N_g)$, the mean curvature function computed with respect to the upward unit normal to the graph in question.

                For a proof of the following two Propositions see Appendix \ref{app_fol}. Here and below, $\delta$ denotes the Euclidean metric in $\overline{C_{\theta}}$.

            		\begin{prop}\label{pro:EuclFol}
            	For every $t\in\R$ there exists a neighbourhood $U_t$ of $\delta\in X$, together with $\varepsilon',\varepsilon''>0$ such that
            	\[
            	\forall \ t'\in (t-\varepsilon',t+\varepsilon'), \ \forall g\in U_t , \ \forall \ w'\in Y \ \|w'\|_Y<\varepsilon''
            	\]
            	there exists a unique $u=u(t',g,w')$ such that $\Phi(t',g,w',u)=(0,0,0)\in Z_1\times Z_2\times Z_3$, i. e. such that the graph of the function $x\mapsto t'+w'(x)+u(x)$ describes a $g$-minimal $C^{2,\alpha}$ surface, meeting $\Gamma_1\times\R$ $g$-orthogonally, and whose values at $\Gamma_2$ are given by $t'+w'(x)$.
           	\end{prop}
           	
           	Notice that by means of a standard argument we can always assume that the conclusion of such proposition holds for $t'\in [-1/2,1/2]$ and $g\in U\subset U_t$.
           	Possibly further shrinking the neighborhood in the domain the construction above provides in fact a local foliation. Indeed, this follows directly from Proposition \ref{pro:EuclFol} once we observe that for $g=\delta, w=0$ the family in question is just given by $u(t',\delta,0)=0$ so that the speed of the variation $t'\mapsto t'+u(t',\delta, 0)$ is the same at all points and equals one and by continuity, the speed of the variation $t'\mapsto t'+w'+u(t',g,w')$ is bounded from below by a positive constant for all $g \in U'\subset U\subset X$ and $\|w'\|_Y<\sigma\varepsilon''$ for some $\sigma\in (0,1)$.

             	\begin{prop}\label{pro:EuclFol2}
           		With the setup as above, the map 
           		\[
           		t\in \left[-1/2,1/2\right] \mapsto \textrm{graph} \ (t+w'+u(t,g,w'))\in \overline{C_\theta}
           		\]
           		determines a foliation of an open region in $\overline{C}_{\theta}$ by free boundary minimal hypersurfaces, with leaves having prescribed height $t+w'$ on the boundary component $\Gamma_2\times\R$, provided $g$ is close enough to the Euclidean metric and $w'$ has small enough norm in $Y$. 
           	\end{prop}

             \begin{rmk}\label{rem:reg}
           	A posteriori, if the background metric $g$ and the function $w$ are actually smooth (i. e. $C^{\infty}$) then so will be the leaves of the foliation in question, by virtue of linear interior and boundary Schauder estimates. Indeed, each minimal graph can locally be described via a defining function that solves a uniformly elliptic linear equation (see e. g. the discussion presented in the proof of statement (2) of Proposition \ref{prop:firstvar}).
           	\end{rmk}

            	Let us now apply such construction to obtain a local foliation by free boundary minimal leaves near a boundary point of our ambient manifold. More precisely, as throughout the article let $(\mathcal{N},g)$ be a compact Riemannian manifold with boundary, let $M\subset \mathcal{N}$ be a properly embedded, free boundary minimal hypersurface and consider a point $p\in\partial M$: we wish to apply the above discussion to foliate a suitably small neighborhood of $p$. 
           	
            	In the setting above, let us choose a local coordinate system near $p$, namely a local diffeomorphism (onto its image) $\Psi:\overline{S_{\theta}}\times [-a,a]\to (\mathcal{N},g)$ and endow the domain $\overline{S_{\theta}}\times [-a,a]$ with the pull-back metric $\Psi^{\ast}g$. Then, given a positive real number $0<r<1$ let
            	\[
            	S^r_{\theta}:=\left\{rx\in\R^{n+1} \ | \ x\in S_{\theta} \right\}, \ \ C^r_{\theta}=S^r_{\theta}\times [-ar,ar]
            	\]
            	and consider the rescaling maps $\Pi_r:\overline{S_{\theta}}\times [-a,a]\to S^r_{\theta}\times [-ar,ar]$. If we let
            	\[
            	g^r:=\Pi_r^{\ast}\left(\Psi^{\ast}g\right)
            	\]
            	then $g^r\to (\Psi^*g)(0)$ as we let $r\to 0$, the convergence being true in any given $C^{k,\alpha}$ norm. Thus, modulo possibly renaming $\Psi$ we can, without loss of generality, assume that in fact $g^r\to \delta$ as we let $r\to 0$ in the sense above. In particular, one can find $r_0>0$ such that for any $0<r<r_0$ the metric $g^r$ belongs to the neighborhood $U'$ for which the conclusions of Proposition \ref{pro:EuclFol2} hold true.
           	Furthermore, we can always assume that the diffeomorphism $\Psi$ above is chosen so that $M\cap \Psi(\overline{S_{\theta}}\times [-a,a])$ can be written as a graph of a smooth function $w'_M$ over the domain $\overline{S_{\theta}}$ with $\|w'_M\|_{Y}$ so small that Proposition \ref{pro:EuclFol2} is applicable. Therefore, in those local coordinates, the graphs of $t+w'_M+u(t,g^r,w'_M)$ when $t\in [-a,a]$ give a foliation near $M$ such that the leaf corresponding to $t=0$ is precisely $M\cap \Psi(\overline{S_{\theta}}\times [-a,a])$.
           	
           	\begin{rmk}\label{c1_reg}
           		We will require the following regularity statement later on: suppose that $v\in C^1(\overline{S_\theta})\cap C^{2,\al}(\overline{S_\theta}\sm \{0\})$ is a free boundary minimal graph, then in fact $v \in C^{2,\al}(\overline{S_\theta})$ and actually it is as smooth as the data allows. One can see this directly as a consequence of the free boundary Allard regularity statement (see Theorem 4.13 in \cite{GJ86B} by Gr\"uter and Jost), but we outline a more direct argument here when we are allowed to assume more about $v$. 
           		
           		Let us suppose (as will be the case along the proof of Theorem \ref{thm_remov_sing}) we are given a free boundary minimal graph $v\in C^1(\overline{S_\theta})\cap C^{2,\al}(\overline{S_\theta}\sm \{0\})$ for which we may assume $\|v\|_{Z_3} + \|v\|_{C^1(\overline{S_\theta})}$ is as small as we like. Letting $w^\prime$ be so that $v|_{\Gamma_2} = w^\prime_{\Gamma_2}$ and $\|w^\prime\|_{Y}\leq \|v\|_{Z_3}$ then we claim that $v\equiv v^\prime:= w^\prime + u(0,g,w^\prime)$ in $S_\theta$. In particular $v\in C^{2,\al}(\overline{S_\theta})$. This follows by considering the difference $v^\prime - v$ which satisfies zero Dirichlet conditions on $\Gamma_2$ and zero Neumann conditions on $\Gamma_1$ (as soon as one chooses appropriate Fermi coordinates, with respect to $\partial\mathcal{N}$, for the given metric $g$ to do this), and finally it solves some uniformly elliptic linear PDE on $S_\theta$ whose first eigenvalue (with respect to free boundary perturbations) can be made positive when $\|v^\prime\|_{C^1(\overline{S_\theta})} + \|v\|_{C^1(\overline{S_\theta})}$ is sufficiently small (which it is). Therefore we conclude $v\equiv v^\prime$. The argument here is just the free boundary analogue of that in \cite[p 255]{Whi87B}, and corresponds to the final statement of the Proposition in the Appendix of the same paper. 
           	\end{rmk}

           	\begin{rmk}\label{rem:comp}
           		Given any $C>1$ one can find $U_C\subset U', \varepsilon'_C\in (0,\varepsilon'), \varepsilon''_C\in (0,\varepsilon'')$ so that
           		\begin{equation}\label{eq:control}
           		\sup_{x\in S_{\theta}}|v_1-v_2|\leq C \sup_{x\in\Gamma_2}|w_1-w_2|
           		\end{equation}
           		provided $g_r\in U_C, |t|<\varepsilon'_C$ and $\|w_i\|_Y<\varepsilon_C''$,
           		where we have set $v_i :=t+ w_i+u(t,g^r,w_i)$ for $i=1,2$.
           		Indeed, such conclusion is true when $g^r=\delta$ (obviously with $C=1$), so the conclusion follows by continuity.
           	\end{rmk}

               Thereby, Proposition \ref{pro:EuclFol2} together with the above discussion implies the following assertion:
               
               \begin{prop}\label{prop:foliation}
               	Let $(\mathcal{N}^{n+1},g)$ be a compact Riemannian manifold with boundary, and let $M\subset \mathcal{N}$ be a properly embedded, free boundary minimal hypersurface. Given a point $p\in\partial M$ there exists $\rho>0$ such that the ambient geodesic ball $B_{\rho}(p)$ is foliated by free boundary minimal leaves $\left\{S_t\right\}, t\in [-1/2,1/2]$ with $S_0=M\cap B_{\rho}(p)$. Furthermore, one can choose a local chart around $p$ so that each slice is described by a coordinate graph as above, and the estimates \eqref{eq:control} are satisfied.
              	\end{prop}

\section{The boundary removable singularity theorem}\label{sec:remov_sing}

\indent In the course of the proof of Theorem \ref{thm_weak_comp} we will need a removal of point singularities result for properly embedded free boundary minimal hypersurfaces with bounded volume and bounded first Jacobi eigenvalue. Such a result relies on the work of Schoen and Simon \cite{SS81}. The argument we present here to tackle point singularities lying on the boundary is modelled on the proof of Theorem 2 in \cite{Whi87B}. \\
\indent This result is patently local, so we will work in subdomains of the unit ball in $\R^{n+1}$ endowed with a Riemannian metric $g$ and suitable rescalings thereof. Considered  $B_r(0)$ the open ball of radius $r>0$ centered at the origin in $\mathbb{R}^{n+1}$, we let $B^{+}_r(0)$ denote the half-ball $\{x \in B_r(0) \ | \  x^1\geq 0\}$.

\begin{thm}\label{thm_remov_sing}
  Let $M_0$ be a smooth, embedded hypersurface in $\overline{B^{+}_1(0)}\setminus\{0\}$, with $0$ in the closure\footnote{To avoid ambiguities, we stress that $\partial M_0$ stands for the boundary of the manifold $M_0$ (in the standard sense of Differential Geometry) and the closure is meant in $\overline{B_1(0)}.$ Similar remarks apply to the statement of Theorem \ref{thm_remov_sing_int}, with straightforward modifications.} of $\partial M_0$ and $M_0\cap  \partial \overline{B^{+}_1(0)}=\partial M_{0}$, that is a free boundary minimal hypersurface at $\de B^{+}_1(0)\cap\{x^1=0\}\setminus \{0\}$ with respect to some Riemannian metric $g$ defined on $\overline{B^+_1(0)}$. \\
\indent Assume that $2\leq n \leq 6$. If $\lambda_1(M_0) \geq -\mu$ for some constant $\mu\geq 0$ and $\h^{n}(M_0) \leq \Lambda$ for some constant $\Lambda>0$, then $M=M_0\cup\{0\}$ is a smooth, embedded minimal hypersurface.  
\end{thm}

\begin{proof}
 \indent Given any sequence $r_i > 0$ converging to zero, we first argue that any blow-up sequence $\{r_i^{-1}M_0\}$ in $(B^{+}_{r_i^{-1}}(0)),g(r_i x))$ has a subsequence converging to some half-hyperplane $H_0$ in the (Euclidean) half-space $(\mathbb{R}^{n+1}_{+},g(0))$. Indeed, under our hypotheses, we can apply the curvature estimate of Theorem \ref{thm_A_bound} and the basic compactness Theorem \ref{thm_local_comp} to conclude that there is a subsequence converging locally smoothly and graphically on compact sets of $\mathbb{R}^{n+1}_+$ away from $0$, with finitely many sheets, to some free boundary minimal hypersurface $M_{\infty}$ in $(\mathbb{R}^{n+1}_{+}\setminus\{0\},g(0))$. Notice that the monotonicity formula, Corollary \ref{cor_mon} was used to guarantee uniform volume bounds on compact subsets of $\mathbb{R}^{n+1}\setminus\{0\}$, and that it implies that the limit $M_{\infty}$ is a half-cone (i.e., it is invariant by dilations with respect to the origin). Also, $M_{\infty}$ must be stable on compact sets of $\mathbb{R}^{n+1}\setminus\{0\}$, by upper semi-continuity of the first Jacobi eigenvalue. Hence, the reflection of $M_{\infty}$ across the hyperplane $\partial \mathbb{R}^{n+1}_{+}$ is a stable minimal cone in $\mathbb{R}^{n+1}$, smooth except possibly at the origin. The claim now follows immediately since stable minimal cones in dimensions $2\leq n \leq 6$ are hyperplanes \cite{S68}. \\
 \indent The fact that the blow-up limits $M_{\infty}$ are totally geodesic implies that the curvature estimates can be improved to say that $|A|(x)|x|\rightarrow 0$ as $x\rightarrow 0$, uniformly on compact subsets of $B^{+}_1(0)\sm \{0\}$. By a well-known Morse-theoretic argument, it follows that $M_0$, in a small neighbourhood of $0$, must be a (finite and disjoint) union of $n$-dimensional discs and half-discs with exactly one critical point, and half-discs of dimension $n$, punctured at $0$, with no critical points for the square of the distance function from the origin. Indeed, the sequence is graphically converging to a half-hyperplane and the curvature estimate eventually implies that the Euclidean norm squared $|x|^2$, restricted to $M_0$, has strictly positive Hessian. Once each punctured free boundary minimal half-disk is proven to be smooth up to the point $0$, the boundary maximum principle (the Hopf boundary point lemma) leads to the conclusion that two of them cannot intersect only at $0$. Without loss of generality, we will therefore assume that $M_0$ is a punctured half-disk in $B^{+}_1(0)\setminus\{0\}$ such $x\in M_0 \mapsto |x|^2 \in \mathbb{R}$ is a function with no critical points. \\
 \indent Notice that, in the above argument, the free boundary half-hyperplane $M_{\infty}$ could possibly depend on the blow-up sequence $r_i^{-1}M$ we have chosen. We will now show that this is not the case. To that scope, let us fix some limiting free boundary half-hyperplane, say $H_0=\R^{n+1}_{+}\cap\{x^{n+1}=0\}$, set $M_{i}=r_{i}^{-1}M$ for each $i\geq 1$ and assume $M_i$ converges to $H_0$ in the sense explained above. In the notation of Section \ref{app:fol}, using the locally graphical convergence to free boundary half-hyperplanes established above and the absence of critical points of $|x|^2$, up to a subsequence (which we shall not rename) we can eventually write $M_{i} \cap (\Gamma_2 \times \mathbb{R})$ as the graph of a function $w_{i}$ on $\Gamma_2$ such that $||w_{i}||_{C^{2,\alpha}}$ converges to zero. Moreover, for all $\varepsilon>0$, $M_{i}\cap (S_{\theta}\times \mathbb{R}) \subset S_{\theta}\times [-\varepsilon,\varepsilon]$ eventually. Comparing $M_{i}\cap C_{\theta}$ with the leaf $L_{i}$ through $0$ of the free boundary minimal foliation whose boundary values, over $\Gamma_2$, are $w_{i}+t$, $t\in [-\varepsilon',\varepsilon']$ (see Proposition \ref{pro:EuclFol} in Section \ref{app:fol}), it follows from the maximum principle that each $M_{i}\cap C_{\theta}$ must eventually be either above or below the leaf $L_{i}$. Thus, any other free boundary half-hyperplane $H_0'$ that arises as the limit of a blow-up sequence $M_{i^{'}}:=r^{-1}_{i'}M$ will be then contained in either of the two closed halfspaces determined by $H_0$ in $\R^{n+1}_{+}$. On the other hand, since $H'_0$ and $H_0$ both contain the point $0$, they must be equal, as we wanted to prove. \\ 
 \indent In conclusion, as $r$ goes to zero, the dilated hypersurfaces $r^{-1}M_0$ converge locally graphically and smoothly away from zero, to a free boundary half-hyperplane $H_0$. In particular, $M_0\cup\{0\}$ can be written, in a neighbourhood of $0$, as a $C^1$ graph over $H_0$ that is smooth except possibly at $0$. Using Remark \ref{c1_reg} (i.e. arguing as in \cite[Appendix]{Whi87B}) we then conclude that $M_{0}\cup \{0\}$, in a small neighbourhood of $0$, is described by the graph of a $C^{2,\alpha}$ function on $S_{\theta}\subset H_0$. At that stage, smoothness follows by standard regularity theory.   
\end{proof}

An \textsl{interior} version of the removable singularity theorem stated above was already employed in \cite{ACS15}, but we state it here for the sake of completeness and for later use. The proof also relies on the work by Schoen-Simon, and follows along the very same lines with straightforward changes.

\begin{thm}\label{thm_remov_sing_int}
	Let $M_0$ be a smooth, embedded hypersurface in $\overline{B_1(0)}\setminus\left\{0\right\}$, with $0$ in the closure of $M_0$ and $M_0\cap\partial \overline{B_1(0)}=\partial M_0$, that is minimal with respect to some Riemannian metric $g$ defined on $\overline{B_1(0)}$. \\
	\indent Assume that $2\leq n \leq 6$. If $\lambda_1(M_0) \geq -\mu$ for some constant $\mu\geq 0$ and $\h^{n}(M_0) \leq \Lambda$ for some constant $\Lambda>0$, then $M=M_0\cup\{0\}$ is a smooth, embedded minimal hypersurface.  
\end{thm}

\section{A convergence result for free boundary minimal hypersurfaces} \label{sec:thm2}

In this section, we prove the following compactness result, which implies Theorem \ref{thm_weak_comp}.

\begin{thm}\label{thm_weak_comp_ext}
	Let $2\leq n\leq 6$ and $(\Nn^{n+1},g)$ be a compact Riemannian manifold with boundary. For fixed $\Lambda, \mu\in \R_{\geq 0}$ and $p\in \N$, suppose that $\{M_k\}$ is a sequence in $\mathfrak{M}_p(\Lambda, \mu)$. Then there exist a smooth, connected, compact embedded minimal hypersurface $M\subset\mathcal{N}$ meeting $\partial\mathcal{N}$ orthogonally along $\partial M$, $m\in \mathbb{N}$ and a finite set $\mathcal{Y}\In M$ with cardinality $|\mathcal{Y}|\leq p-1$  such that, up to subsequence, $M_k \to M$ locally smoothly and graphically on $M\sm \mathcal{Y}$ with multiplicity $m$.  Furthermore, if $M\in\mathfrak{M}$ then $M\in\mathfrak{M}_{p}(\Lambda,\mu)$. 
\end{thm}

\begin{proof}
Let $\mu_k$ be the measure defined on open sets $U$ of $\mathcal{N}$ by
\begin{equation*}
   \mu_k(U) = \h^{n}(M_k\cap U)
\end{equation*}
\noindent for each integer $k$. Since the volume of the hypersurfaces $M_k$ is uniformly bounded, these are Radon measures and we can assume, passing to a subsequence if necessary (but without renaming), that $\mu_k$ converges weakly to a Radon measure $\mu_{\infty}$. \\
\indent A point $x$ in the support of the measure $\mu_{\infty}$ will be called a point of good convergence if, possibly extracting a further subsequence, the hypersurfaces $M_k$ are converging locally graphically and smoothly with finite multiplicity to a free boundary minimal hypersurface in a neighbourhood of $x$, and a point of bad convergence otherwise. \\
\indent As a trivial consequence of the definitions above, each point of good convergence belongs to a free boundary embedded minimal hypersurface contained in the support of $\mu_{\infty}$. \\
\indent Let $\mathcal{Y}$ denote the set of points of bad convergence. In view of the curvature estimates given in Theorem \ref{thm_A_bound} and the convergence result provided by Theorem \ref{thm_local_comp}, it is possible to check that 
\begin{equation*}
  \mathcal{Y} \subset \{x \in \mathcal{N} \ | \ \text{ for all $\varepsilon > 0$, } \limsup \lambda_{1}(M_k\cap B_{\varepsilon}(x)) < -\mu \}.
\end{equation*} 
\indent This implies that $\mathcal{Y}$ has at most $p-1$ points. Indeed, if $\mathcal{Y}$ contained $p$ distinct points $x_1,\ldots,x_p$, one would find $p$ disjoint balls $B_1,\ldots,B_p$ where, eventually (and up to a subsequence), $\lambda_{1}(M_k\cap B_i) < - \mu$ for all integers $k$ and any given $i$, which contradicts the hypothesis that $\lambda_{p}(M_k) \geq - \mu$ (see the discussion presented in Subsection \ref{subs:spectrum}).  \\
\indent Therefore, the support of the measure $\mu_{\infty}$ is the disjoint union of the finite set $\mathcal{Y}$ and a smooth embedded free boundary minimal hypersurface, which we will denote by $M_{0}$. Notice that no point of $\mathcal{Y}$ is isolated (because $\mu_{\infty}$ is the limit of the Radon measures associated to the free boundary minimal hypersurfaces $M_k$, each of them satisfying the monotonicity formula). We will now show that the closure of $M_0$ (which coincides with the support of $\mu_{\infty}$) is actually a smooth, free boundary minimal hypersurface. \\ 
\indent To that scope, let us fix a point $y\in\mathcal{Y}$ and distinguish these three cases:
\begin{enumerate}
\item{$\liminf_{k\to\infty} \textrm{dist}_{(\mathcal{N},g)}(y,\partial M_k)\geq d_0>0$;}	
\item{there exists a sequence $\left\{y_k\right\}\subset\mathcal{N}, \ y_k\in\partial M_k, \ y_k\to y$ and $(\partial M_0\setminus\left\{y\right\})\cap B_{\rho}(y)\neq\emptyset$  for all $\rho>0$;}
\item{there exists a sequence $\left\{y_k\right\}\subset\mathcal{N}, \ y_k\in\partial M_k, \ y_k\to y$ and $(\partial M_0\setminus\left\{y\right\})\cap B_{\rho_0}(y)=\emptyset$ for some $\rho_0>0$.}	
\end{enumerate}	Also, let us choose once and for all a constant $\varepsilon_0>0$ such that $B_{\varepsilon_0}(y)\cap\mathcal{Y}=\left\{y\right\}$.\\

\indent As far as case 1) is concerned, we can proceed as follows. If $y\notin\partial\mathcal{N}$, a potential \textsl{interior} singularity of $M_0$, we can just invoke Theorem \ref{thm_remov_sing_int} to conclude smoothness at the point in question. Indeed, we claim that there is some $0<\eps< \min\left\{d_0/3,\varepsilon_0\right\}$ with $\lambda_1(M_0\cap B_\eps(y)) \geq - \mu$: if this were not the case, one would find $p$ disjoint open sets $U_i$ of the form $B_{2\rho}(y)\setminus B_{\rho}(y)$, $\rho>0$, where $\lambda_{1}(M_0\cap U_i) < - \mu$, and of course the smooth convergence of $M_k$ together with the upper semi-continuity of the first Jacobi eigenvalue would then imply $\lambda_{1}(M_k\cap U_i) < - \mu$ for every $i\in I, |I|=p$ and large enough $k$, hence $\lambda_p(M_k)<-\mu$, a contradiction. 
If instead $y\in\partial\mathcal{N}$, let us consider an enlargement of $(\mathcal{N},g)$ to a compact Riemannian manifold $(\mathcal{N}',g')$ without boundary (that is to say: $(\mathcal{N},g)$) isometrically embeds in $(\mathcal{N}',g')$ as a regular subdomain).
Then, by virtue of the same argument we have just presented, there must be a small enough $0<\varepsilon< \min\{d_0/3,\varepsilon_0\}$ such that, regarding $M_0\cap B_{\varepsilon}(y)$ as a minimal hypersurface in $(\mathcal{N}',g')$, we have $\lambda'_1(M_0\cap B'_{\varepsilon}(y))\geq -\mu$ (here it is understood, given our notation, that we view the hypersurfaces as lying in $(\mathcal{N}',g')$ and consider unrestricted variations in the enlarged ambient manifold when discussing their local Jacobi spectrum; notice that the eigenvalues $\lambda'_1(M_k\cap B'_{\varepsilon}(y))$ and $\lambda_1(M_k\cap B_{\varepsilon}(y))$ coincide, because each $M_k$ is properly embedded in $\mathcal{N}$ by hypothesis). Thus, also in this case we can directly invoke Theorem \ref{thm_remov_sing_int} again to gain smoothness at $y$. \\

\indent For case 2), we need a preliminary remark. 
Choosing $\varepsilon\in (0,\varepsilon_0)$ small enough, we can assume that $M_0\cap B_{\varepsilon}(y)$ consists of a finite number of connected components, whose closures may only intersect at the point $\left\{y\right\}$ (the finiteness follows by the monotonicity formula). 
 Now, let $M^{\ast}_0\cap B_{\varepsilon}(y)$ be one such component. Notice that it must be a smooth locally graphical limit (away from the single point $y$) of components of $M_k\cap B_{\varepsilon}(y)$. If $y$ does not belong to the closure of $\partial M^{\ast}_0\cap B_{\varepsilon}(y)$ then one can argue as for case 1) to see that $y$ is a removable singularity in the interior of $M^{\ast}_0$. Thereby we are left with the case when $y$ is in the closure of $\partial M^{\ast}_0\cap B_{\varepsilon}(y)$. For $k$ large enough we exploit the smooth graphical convergence up to and including the boundary (away from the sole point $\left\{y\right\}$) to gain the very same lower bound on the first eigenvalue, $\lambda_1(M^{\ast}_0\cap B_\eps(y)) \geq - \mu$. Theorem \ref{thm_remov_sing} then implies that such $y$ is a removable singularity of the boundary of $M^{\ast}_0$. A posteriori, we conclude there is only one such component $M^{*}_{0}$ (by embeddedness of $M_0$, in the situation we are analysing it is impossible to find at the same time two components corresponding to the two different subcases described above). Thus, in case (2), we conclude that $y$ is a boundary removable singularity of $M_0$. \\

\indent Lastly, for case 3) we claim that $\lambda'_1(M_0\cap B'_\eps(y)) \geq - \mu$ for $\varepsilon\in (0,\varepsilon_0)$ sufficiently small. Assume this were false, then the usual argument would lead to finding positive numbers $\varepsilon_1>\varepsilon_2>\ldots>\varepsilon_{p+1}$ (with $\varepsilon_1<\rho_0$) so that, set $U'_i=B'_{\varepsilon_i}(y)\setminus B'_{\varepsilon_{i+1}}(y)$ one would have $\lambda'_1(M_0\cap U'_i)<-\mu$. By assumption, $(\partial M_0\setminus\left\{y\right\})\cap B_{\rho_0}(y)=\emptyset$. Thus, for $k$ sufficiently large $\partial M_k\cap B_{\rho_0}(y)\subset B_{\varepsilon_{p+1}}(y)$ and necessarily $\partial M_k\cap U_i=\emptyset$ for every choice of $i=1,\ldots, p$, where we denote $U_i=B_{\varepsilon_i}(y)\setminus B_{\varepsilon_{i+1}}(y)$. But then, possibly taking $k$ even larger we would have that $\lambda_1(M_k\cap U_i)<-\mu$ for every $i=1,\ldots, p$, which is impossible since $\left\{M_k\right\}\subset\mathfrak{M}_p(\Lambda,\mu)$. Hence, Theorem \ref{thm_remov_sing_int} applies.\\

\indent Thus, it follows from the above analysis that the closure of $M_{0}$, denoted henceforth by $M$, is a smooth compact embedded minimal hypersurface. $M$ must be connected as $\mu_{\infty}$ is the weak limit of connected embedded free boundary minimal hypersurfaces (by monotonicity). It follows  that the continuous and locally constant function that assigns to each point $x$ in $M_0$ the number of graphical sheets converging to it, is a constant integer $m\geq 1$ (actually, this shows that $\mu_{\infty}$ is the measure defined by $m \h^{n}(M\cap U)$ on open sets $U$ of $\mathcal{N}$). \\
\indent If all hypersurfaces $M_k$ have volume bounded by a constant $\Lambda$, then $m\h^{n}(M)$ will also be bounded by $\Lambda$. If it happens that $M$ is properly embedded, a standard cutoff argument ensures that $\lambda_p(M)\geq \lambda_{p}(M_0)\geq -\mu$, and thereby the proof that in this case $M\in\mathfrak{M}_{p}(\Lambda,\mu)$ is completed.\\
\end{proof}

\section{Stability and degeneration analysis for limit free boundary minimal hypersurfaces}\label{sec:thm4}

Here we shall present the proof of Theorem \ref{thm_mult_anal}.

\begin{proof}

\

\

\noindent\textbf{The case $m=1$.}\\
\indent First of all, we need to see that $\mathcal{Y}$ must be empty. Since we are assuming that $M$ is properly embedded, $\de M = M\cap \de \Nn$, then for any $\eps >0$ there is some uniform $r_0>0$ so that  
\[
\text{$\frac{\h^n(M\cap B_r(p))}{\om_n r^n} \leq \frac12 (1+\eps)$ for all $p\in M\cap \de \Nn$, $r<r_0$}
\]
and analogously for $p\in M\setminus\partial M$ with a constant $1+\varepsilon$ on the right-hand side. Hence, the same inequalities must hold (with a marginally worse constant) for $M_k$ (provided $k$ is large enough) by standard facts about convergence of measures.
The conclusion is now a consequence of the smooth version of Allard's regularity, Theorem \ref{thm_bound_allard}, for free boundary minimal hypersurfaces.

In the case of two-sided limits and $\mathcal{Y}=\emptyset$ then $m=1$  follows by connectedness of $M_k$ (if $\mathcal{Y}=\emptyset$ then the approaching $M_k$ can be written globally as graphs over $M$ - but if the multiplicity is then greater than one we have a contradiction).
 
When we have multiplicity one convergence, the construction of a non-trivial Jacobi field on the limit follows along similar lines to the argument below for higher multiplicities, so we omit the proof. \\

\noindent\textbf{The case $m\geq 2$ and $M$ is one-sided.}\\ 
\indent Suppose that the limit $M$ is one-sided and let $f:\widetilde{M}\to \Nn$
 be the two-sided minimal immersion associated to $M$. 

We change the picture slightly in order to reduce the discussion to analysing two-sided limits. In order to visualise the construction it is best to picture the pulled-back bundle, $f^*NM$, which is trivial (by definition of $\widetilde{M}$). The zero section of this bundle describes $\widetilde{M}$, and a sufficiently small neighbourhood of the zero section, denoted $\widetilde{U}$, is in a two-to-one correspondence with a small tubular neighbourhood $U$ of the one-sided hypersurface $M$ in $\mathcal{N}$. Therefore, we may pull back the metric on $U$ and see $\widetilde{M}\emb\widetilde{U}$ as a two-sided properly embedded free boundary minimal hypersurface. For large enough $k$ (so that eventually $M_k$ lies inside this tubular neighbourhood of $\widetilde{M}$) we can equally consider the pull back of $M_k$, denoted $\widetilde{M}_k\emb \widetilde{U}$ which is again an embedded minimal hypersurface (possibly disconnected, but with at most two components). Nevertheless we still have $\widetilde{M}_k \to m\widetilde{M}$ locally smoothly and graphically on $\widetilde{M}\sm\widetilde{\mathcal{Y}}$ with $|\widetilde{\mathcal{Y}}|=2|\mathcal{Y}|$.

By the analysis of the previous case, it is easy to check that $\mathcal{Y}=\emptyset$ (together with $m\geq 2$) gives that $\widetilde{M}_k$ is disconnected (having two components) which implies that $m=2$ and $M_k \simeq \widetilde{M}$ eventually. 

Lastly, if we produce a (positive) Jacobi field on (the two-sided) $\widetilde{M}\emb \widetilde{U}$ then this corresponds to a (positive) Jacobi field on the immersion $\widetilde{M}\In \Nn$. In particular $\widetilde{M}$ must be stable, and since the first eigenfunction appears with multiplicity one, then $nullity(\widetilde{M})=1$. This in turn gives that $M$ must be strictly stable, that is to say $\lambda_1(M)>0$. Indeed, if that were not the case one could find an odd eigenfunction for the Jacobi operator of $\tilde{M}$ of eigenvalue $\lambda_1(M)$ so that $\lambda_1(M)\geq 0$ but on the other hand it cannot be $\lambda_1(M)=0$ for that would imply such function would in fact be a first eigenfunction for $\tilde{M}$, which is impossible given that it is sign-changing.\\

\noindent\textbf{The case $m\geq 2$ and $M$ is two-sided.}\\
\indent We will construct a positive Jacobi function over the limit in this case, which gives, by the same argument presented in the previous paragraph, that $nullity(M) =1$ and $M$ stable.  

Let $N$ be a choice of global unit normal and $X\in \mathfrak{X}_\de$ be an arbitrary extension of $N$ to $\Nn$ (so that $X(p)\in T_p \de \Nn$ for all $p \in \de \Nn$). Now let $\Phi(x,t)$ be the one-parameter family of diffeomorphisms associated with $X$, so that $\pl{\Phi}{h}(x,h) = X(\Phi(x,h))$. For a set $V\subset M$ we let $V_\dl$ be the $\dl$-thickening of $V$ with respect to $\Phi$ so that 
$$\text{$V_\dl := \{\Phi(x,h) :$ $x\in V$ and $|h|<\dl$\}}.$$ 
By assumption, on each $\Om \cemb M\sm \mathcal{Y}$ there must exist a family of $m$ functions which we order by height $\{u_k^1 < \dots < u_k^m\} \in C^{\infty}(\Om, \R)$ so that for $x\in \Om$,
$$M_k\cap \Om_\dl = \{\Phi(x,u_k^1(x)),\dots, \Phi(x,u_k^m(x)), x\in\Omega\}.$$
Setting $v_k(x,t) = tu^m_k(x) + (1-t)u^1_k(x)$ for $t\in [0,1]$ we denote $\Phi^k (x,t) = \Phi^k_t(x)= \Phi(x,v_k(x,t))$ and 
$$\Sigma_k(t) = (\Phi^k_t)_{\sharp}(\Om) = \{\Phi(x,v_k(x,t)), x\in \Om\}$$
with 
$$X_k(\Phi^k_t(x)):= \pl{\Phi^k_t}{t}(x) = (u^m_k(x)-u^1_k(x))X(\Phi^k_t(x)).$$
Note that $\Sigma_k(1)$ is the top leaf of $M_k$ over $\Om$ and $\Sigma_k(0)$ is the bottom leaf of $M_k$ over $\Om$ (and therefore are both free boundary and minimal) and $\Sigma_k(t)$ is a smooth one parameter family of hypersurfaces with boundary connecting the two.

Now consider any compactly supported ambient vector field $Z\in \mathfrak{X}_\de$. This gives rise to variations of $\Sigma_k(t)$, denoted 
\begin{eqnarray*}
\Sigma_k(t,s)
=(\Psi_s)_\sharp(\Sigma_k(t)) 
\end{eqnarray*}
where $\Psi_s$ is a family of diffeomorphisms induced by $Z$ i.e. $Z(x)=\pl{\Psi_s(x)}{s}\vlinesub{s=0}$. We have 
$$\left.\pl{}{s}\h^n(\Sigma_k(t,s))\right|_{s=0} = \int_{\Sigma_k(t)} div_{\Sigma_k(t)}(Z)\id\h^n$$
and this is a smooth function of $t$ by the definition of $\Sigma_k(t)$. We also know that this quantity is null when $t=0,1$ for all $k$, thus 
$$\int_0^1 \pl{}{t}\left.\pl{}{s} \h^n(\Sigma_k(t,s))\right|_{s=0}\id t = 0.$$ 

Therefore we have, by Appendix \ref{app:2ndvar}

\begin{align*}
0=&\int_0^1 \pl{}{t}\left.\pl{}{s}\h^n(\Sigma_k(t,s))\right|_{s=0} \id t\nonumber\\
=& \int_{0}^{1} \left( \int_{\Sigma_k(t)} \la\Db (X^\bot) , \Db (Z^\bot) \ra - Ric_\mathcal{N} (X^\bot, Z^\bot) - |A|^2\la X^\bot , Z^\bot \ra \id \h^n \right. \\
 & \quad \quad \left. + \int_{\de \Sigma_k(t)} \la \nabla_{X^{\perp}} Z^{\perp}, \nu_e \ra \id \h^{n-1}+ \int_{\Sigma_k(t)}\Xi_1(X,Z,H) \id \h^n + \int_{\partial \Sigma_k(t)} \Xi_2(X,Z,H,\nu,\nu_e) \id \h^{n-1}\right)  dt.\label{2ndv}
\end{align*}

By pulling everything back to $\Om$ and assuming that $Z|_{\Om} = \eta N$ where $N$ is as above and $\eta$ is smooth we have (setting $\ti{h}_k = u^m_k - u^1_k$)
 \begin{eqnarray*}
0 &=&\int_0^1\left[\int_{\Om} \D \ti{h}_k \cdot \D \eta - \ti{h}_k\eta(|A|^2 + Ric_\mathcal{N} (N, N))  + \ti{W}_k(t)(\ti{h}_k,\eta)\id \h^n\right. \nonumber\\
&& +\left.\int_{\de \Nn \cap \Om} {\rm{II}}(N,N)\ti{h}_k\eta + \ti{w}_k(t)(\ti{h}_k, \eta) \id \h^{n-1} \right]\id t
\end{eqnarray*}
where $\ti{W}_k$ is linear and at most first order in its arguments and whose coefficients go to zero smoothly in $k$, equally $\ti{w}_k$ is first order with coefficients going to zero smoothly. 
Letting $W_k = \int_0^1 \ti{W}_k (t) \id t$ and similarly $w_k = \int_0^1 \ti{w}_k(t)\id t$ by Fubini's theorem we are left with 
 
\begin{eqnarray}\label{eq_approx}
0 &=&\int_{\Om} \D \ti{h}_k \cdot \D \eta - \ti{h}_k\eta(|A|^2 + Ric_{\mathcal{N}} (N, N))  + W_k(\ti{h}_k,\eta)\id \h^n \nonumber\\
&& +\int_{\de \Nn \cap \Om} {\rm{II}}(N,N)\ti{h}_k\eta + w_k(\ti{h}_k, \eta) \id \h^{n-1}.
\end{eqnarray}

\textbf{Claim 1:} Fixing $z\in \Om\setminus\partial M$ and $h_k(x):= \ti{h}_k(z)^{-1} \ti{h}_k(x)$ we have that $h_k$ is bounded in $C^l$ for all $l$ on any compact subset of $\Om$. \\

Assuming the claim we can take a subsequence so that $h_k \to h$ smoothly and using \eqref{eq_approx} we end up with $h$ being a solution to 
$$\twoparteq{-\Dl_M h - (|A|^2 + Ric_{\Nn}(N,N)) h = 0}{\text{in $\Om$}}{\pl{h}{\nu} = -{\rm{II}}(N,N)h}{\text{on $\de M\cap \Om$.}}$$
Since $h(z) =1$ and $h\geq 0$ we have in fact that $h>0$ on the interior of $\Om$ by maximum principle. Thus, by taking an exhaustion of $M$ by sets $\Om\cemb M\sm \mathcal{Y}$ we end up with a solution $h:M\sm\mathcal{Y}\to \R$ to the Jacobi equation and $h>0$ on $M\sm \mathcal{Y}$.\\

\textbf{Claim 2:} $h$ is uniformly bounded and therefore extends to a smooth solution of 
$$\twoparteq{-\Dl_M h - (|A|^2 + Ric_{\Nn}(N,N)) h = 0}{\text{in $M$}}{\pl{h}{\nu} = -{\rm{II}}(N,N)h}{\text{on $\de M$.}}$$
By the interior and boundary versions of the maximum principle, the function $h$ must be positive everywhere.

\

It remains to prove the two claims.\\

\textbf{Proof of Claim 1:} We will first show that $h_k$ is uniformly bounded on compact subsets of $\Om$. Once we know this then standard elliptic regularity theory (and bootstrapping) will allow us to upgrade a $C^0$ estimate to $C^l$ control for any $l$ - i.e. we can apply elliptic regularity theory with (oblique) boundary conditions (see e.g. \cite[Lemma 6.29]{gt}) to obtain smooth control on $\ti{h}_k$ for any $V \cemb \Om$.

To see that $h_k$ is bounded we first note that for any $V\cemb \Om\sm \de \Nn$, $h_k$ is a positive solution to an equation with uniform control on the ellipticity - thus the usual Harnack estimate yields the existence of some $C=C(V)$ so that 
$$\sup_V h_k \leq C h_k(z) = C.$$

For a contradiction, suppose that such an estimate fails for $V\cemb \Om$ with $z\in V$. In this case let $x_k$ be so that $\sup_V h_k = h_k(x_k)\to \infty$, and therefore by the interior Harnack estimate we must have $x_k\to x_\infty \in \de M$. By setting
$$f_k (x) = h_k(x_k)^{-1} h_k(x)$$ we now have that $f_k$ is uniformly bounded, thus as above $f_k$ is smoothly controlled and converges to a solution $f$ of $$\twoparteq{-\Dl_M f - (|A|^2 + Ric_{\Nn}(N,N)) f = 0}{\text{in $V$}}{\pl{f}{\nu} = -{\rm{II}}(N,N)f}{\text{on $\de M\cap V$.}}$$
But now $f(z) = 0$ and $f\geq 0$ so we must have $f\equiv 0$ by the maximum principle, which contradicts $f(x_\infty) =1$ and we have proved Claim 1. \\

\textbf{Proof of Claim 2:}
If $y\in\mathcal{Y}$ is in the interior of $M$ then the proof is exactly as in \cite[Claim 6]{Sha15} (or, one could adapt the argument below using an interior foliation trick, for a shorter version). We deal with the case that $y\in \de M$;  with the notation as in Section \ref{app:fol}, take coordinates about $y$, i.e. $S^r_{\theta}(y)\In M$ and extend these to $C^r_{\theta}\In \Nn$ via the flow $\Phi$ with respect to $X$. We will assume that $r$ and $\theta$ are sufficiently small so that we can apply Proposition \ref{pro:EuclFol2}. On $\Gamma_2$ and for $k$ sufficiently large we can make the functions $u_k^1, u_k^m$ as small as we like in any smooth norm, and we take two local foliations with respect to $w_k^1, w_k^m$ which are two functions on $S^r_\theta$ so that $\|w_k^i\|_{Y}\leq \|u_k^i\|_{C^{2,\al}(\Gamma_2)}$ and $w_k^i|_{\Gamma_2} = u^i_k$. We claim that $M_k\cap C^r_{\theta}$ lies between the graphs $w_k^i + u_k(0,g,w_k^i)$ - and once we prove this we will then have, wherever $\ti{h}_k$ is defined $D_k\In S^r_{\theta}$, some uniform $C>1$ (by Remark \ref{rem:comp}) so that
$$ \sup_{D_k} \ti{h}_k \leq C \sup_{\Gamma_2} \ti{h}_k$$
and we can conclude that each $h_k$ is uniformly bounded by a constant independent of $k$, wherever it is defined; thus $h$ is bounded and we are done.  

We will check that $M_k\cap C^r_{\theta}$ lies beneath the graph $w_k^m + u_k(0,g,w_k^m)$. Let $v_{k,t}=t+w_k^m + u_k(t,g,w_k^m)$ be the local foliation. Notice that by taking $r$ sufficiently small and $k$ sufficiently large (so that the Hausdorff distance from $M_k$ to $M$ is uniformly small), Proposition \ref{pro:EuclFol2} guarantees that we can foliate for $t\in [-\fr{r}{2},\fr{r}{2}]$. Now suppose to the contrary that for $t>0$ the graph of $v_{k,t}$ intersects $M_k$. Letting $T$ be the largest $t$ such that this intersection is non-empty yields that the $v_{k,T}$ touches $M_k$ tangentially at a boundary or interior point (\emph{not} along $\Gamma_2$) and lies completely on one side. By the interior and boundary maximum principle we have that $v_{k,T}$ coincides with  $M_k$ which contradicts that $y$ is a point of bad convergence. This completes the proof. 

\end{proof}

\section{A bumpy metric theorem for free boundary minimal hypersurfaces}\label{sec:bumpy}

  In this section, we prove that free boundary minimal hypersurfaces are \textsl{generically} non-degenerate, meaning that they do not have Jacobi fields for a generic choice of the Riemannian metric on the ambient manifold. The precise statement of our result is given in the introduction, see Theorem \ref{thm:bumpy}. In that respect we shall add a few remarks and clarifications.

  \begin{rmk} \label{rmk:cqtop}
  	There we have denoted by $\Gamma^{q}$ the set of $C^{q}$ metrics on $\mathcal{N}$ endowed with the $C^{q}$ topology.
  	$\Gamma^{q}$ has the structure of an open cone inside a complete metric space, in fact a smooth Banach space. Also, recall that the $C^{\infty}$ topology on $\Gamma^{\infty}$ is the smallest topology that makes the inclusions $\Gamma^{\infty} \subset\Gamma^{q}$ continuous for all finite $q$.
  \end{rmk}
  
  \begin{rmk}\label{rmk:Baire}
  	A subset of a complete metric space is called \textit{comeagre} when it contains a countable intersection of open dense subsets (or, equivalently, if its complement is \textit{meagre} or a \textit{first category set} in the sense that it is contained in the union of countably many closed sets with empty interior). By Baire theorem, comeagre subsets are dense.
  \end{rmk}
  
  An important tool in proving Theorem \ref{thm:bumpy} is the following infinite-dimensional version of Sard's lemma due to S. Smale:
  
  \begin{thm}[Cf. Theorem 1.2 in \cite{Sma65}]\label{thm:Sma}
  	Let $X, Y$ be connected Banach manifolds, with $X$ having a countable base, and let $F:X\to Y$ be a Fredholm map of class $C^r$ with $r>\max\left\{\textrm{Ind}(F), 0\right\}$. Then the set of regular values for $F$ is comeagre in $Y$.
  \end{thm}
  
  \begin{rmk}
  	A map $F:X\to Y$ is said Fredholm if for each $x\in X$ the linear map $dF:T_{x}X\to T_{F(x)}Y$ is a Fredholm operator. In such case, the index of $F$, denote by $\textrm{Ind}{(F)}$ is the index of $dF(x)$ for some $x$ (when $X$ is connected, this value shall not depend on the point $x$).	
  \end{rmk}
  
  We have explicitly stated this well-known result in order to stress the role of separability of the domain, for indeed if $X$ is metrizable then the assumption that its topology has a countable basis is equivalent to being separable. On the other hand, we will work with spaces of maps with finite regularity (in fact with H\"older maps, to gain suitable Schauder estimates up to the boundary) to benefit from the Banach manifold structure of these spaces.  These facts motivate the setup we are about to present.\\
  
  \indent For $3\leq j+1\leq q, \ \alpha\in (0,1)$ and a fixed compact, connected, smooth manifold $M^n$ with non-empty boundary, we consider the space of maps $C^{j,\alpha}(M,\mathcal{N})$. For $w\in C^{j,\alpha}(M,\mathcal{N})$ we let 
  \[
  [w]:=\left\{w\circ\varphi \ | \ \varphi\in \textrm{Diff}(M) \right\}
  \]
  namely the set of all reparametrizations of the same map modulo smooth diffeomorphisms of $M$. At this stage, we shall consider the \textsl{set} $\mathcal{PE}^{j,\alpha}=\mathcal{PE}^{j,\alpha}(M,\mathcal{N})$ consisting of the equivalence classes $[w]$ for $w:M\to \mathcal{N}$ a \textsl{proper embedding}, and hence define
  \[
  \mathcal{S}^{q,j,\alpha}:= \left\{(\gamma,[w])\in\Gamma^q\times \mathcal{PE}^{j,\alpha} \ | \ w \ \textrm{is} \ \gamma-\textrm{stationary} \ \textrm{and free boundary} \right\}.
  \]
  Observe that, if we denote by $d_{j,\alpha}:C^{j,\alpha}(M,\mathcal{N})\times C^{j,\alpha}(M,\mathcal{N})\to\R$ a metric inducing the topology on $C^{j,\alpha}(M,\mathcal{N})$ we can set 
  \[
  d([w_1],[w_2]):=\inf_{\tilde{w}_1\in [w_1],\tilde{w}_2\in [w_2]} d_{j,\alpha}(\tilde{w}_1,\tilde{w}_2)
  \]
  and this is easily seen to be a distance on the quotient $\mathcal{PE}^{j,\alpha}$, compatible with its topology (which is Hausdorff). Hence, $\mathcal{S}^{q,j,\alpha}$ inherits a topology as a subset of the product $\Gamma^q\times \mathcal{PE}^{j,\alpha}.$
  
   \begin{rmk}\label{rmk:regul}
   	By elliptic regularity theory, a free boundary minimal immersion $w\in C^{j,\alpha}(M,\mathcal{N})$  into $(\mathcal{N},\gamma)$, where $\gamma\in \Gamma^{q}$, is of class $C^{q,\beta}$ for any $\beta\in (0,1)$. This simple remark implies that the \textsl{set} $\mathcal{S}^{q,j,\alpha}$ does not depend on $j,\alpha$, even though its topology does. Furthermore, it follows that if $(\gamma,[w])\in\mathcal{S}^{q,j,\alpha}$ then the Jacobi fields on $M$ will be of class $C^{k,\beta}$ for all $k< q$ and $\beta\in (0,1)$. 
   \end{rmk}	
   
  \indent At this stage, we let $\pi^{q,j,\alpha}_{\mathcal{S}}:\mathcal{S}^{q,j,\alpha}\to\Gamma^q$ be the projector onto the first factor, i. e. $\pi^{q,j,\alpha}_{\mathcal{S}}(\gamma,[w])=\gamma$. As was first observed by White \cite{Whi87A}, this map encodes the relevant information about degenerate minimal hypersurfaces. The precise content of this assertion is provided by the following `structure theorem', of independent interest.

  \begin{thm}\label{thm:struct}
  	Let $M^n, \ \mathcal{N}^{n+1}$ be $C^{\infty}$ compact manifolds with boundary and, correspondingly, let $\Gamma^q, \mathcal{S}^{q,j,\alpha}$ and $\pi^{q,j,\alpha}_{\mathcal{S}}:\mathcal{S}^{q,j,\alpha}\to\Gamma^q$ be defined as above. Then $\mathcal{S}^{q,j,\alpha}$ is a separable Banach manifold of class $C^{q-j}$ and $\pi^{q,j,\alpha}_{\mathcal{S}}$ is a $C^{q-j}$ Fredholm map of Fredholm index 0. Furthermore, given any $(\gamma,[w])\in\mathcal{S}^{q,j,\alpha}$ the nullity of $w$ equals the dimension of the kernel of the linear map $D \pi^{q,j,\alpha}_{\mathcal{S}}(\gamma,[w])$ (so that, in particular, $w(M)$ admits a non-trivial Jacobi field if and only if 
  	the point $(\gamma,[w])$ is critical for $\pi^{q,j,\alpha}_{\mathcal{S}}$). 	
  \end{thm}

  The next subsection is devoted to the proof of Theorem \ref{thm:bumpy} given Theorem \ref{thm:struct}, the latter being the object of Subsection \ref{subs:struct}.
  
  \begin{rmk}
  	For given $n\geq 2$ there are only countably many diffeomorphisms types of compact, connected $n$-manifolds with boundary, for well-known results dating back to Cairns and Whitehead (see \cite{Cai35, Whi40}) ensure that a differentiable manifold admits a (essentially unique) piecewise linear structure. Therefore, it is enough to prove Theorem \ref{thm:bumpy} considering only embeddings of a given compact, connected, smooth manifold $M$.
  \end{rmk}
  
   \begin{rmk} \label{rmkdiag}
  	Denoting by $\iota$ obvious inclusions, the following diagrams of continuous maps is commutative:
  	
  	\begin{equation*}
  	\begin{CD}
  	\mathcal{S}^{q,q-1,\alpha}    @> \pi^{q,q-1,\alpha} >>  \Gamma^{q} \\
  	@A \iota AA                                                     @A \iota AA  \\
  	\mathcal{S}^{q+1,q,\alpha} @> \pi^{q+1,q,\alpha} >>  \Gamma^{q+1}.  \\
  	\end{CD}
  	\end{equation*}
  	In particular, for example, if $U^{q}$ is an open subset of $\mathcal{S}^{q,q-1,\alpha}$ restricted to which the map $\pi^{q,q-1,\alpha}$ is a homeomorphism onto its image $\pi^{q,q-1,\alpha}(U^q)$, then $U^{q+1}:= \iota^{-1}(U^{q})$ is an open subset of $\mathcal{S}^{q+1,q,\alpha}$ restricted to which $\pi^{q+1,q,\alpha}$ is a homeomorphism onto its image $\pi^{q+1,q,\alpha}(U^{q+1})$, since
  	$\pi^{q+1,q,\alpha}(U^{q+1})  = \iota^{-1}(\pi^{q,q-1,\alpha}(U^q))$. Similar elementary observations will be repeatedly used in the next subsection.
  \end{rmk}
  
  We also need to add a simple but useful observation concerning the dependence on the parameters $q,j,\alpha$ of the objects defined above. As far as Subsection \ref{subs:bumpy} is concerned it is in fact enough to choose $j=q-1$, which we will always do without further comments. Also, the constant $\alpha\in (0,1)$ is understood to be fixed once and for all, for any space we shall deal with. For simplicity of notation, let us write $\mathcal{S}^{q} = \mathcal{S}^{q,q-1,\alpha}(M)$ for each $q\geq 3$ and denote by $\pi^{q} : \mathcal{S}^{q} \rightarrow \Gamma^{q}$ the projection map considered in the Structure Theorem.
  
  \subsection{Proof of Theorem \ref{thm:bumpy}, the bumpy metric theorem}\label{subs:bumpy}
  
  We shall establish the bumpy metric theorem stated above following closely the arguments of  \cite{Whi15} (see, in particular, the proof of Theorem 2.9). Let us first stress an important point.
  
  \begin{rmk} \label{rmk:fund}
  	It is a straightforward consequence of Remark \ref{rmk:regul} that $\mathcal{B}^{q+1}=\mathcal{B}^{q}\cap\Gamma^{q+1}$ and $\mathcal{B}^{\infty}=\mathcal{B}^{q}\cap\Gamma^{\infty}$ for all integers $q\geq 3$.
  \end{rmk}
  
  Let us now present the argument in question.
  
  \begin{proof}
  	Let us fix an arbitrary, compact, smooth manifold with boundary $M$. Given $q\geq 3$ or $q=\infty$, we define $\mathcal{B}^{q}(M)$ to be the set of all $\gamma\in \Gamma^{q}$ such that no free boundary minimal proper embedding $w : M \rightarrow \mathcal{N}$ of class $C^q$, and no free boundary minimal proper immersion of class $C^q$ of the form $w\circ \pi : \widetilde{M} \rightarrow \mathcal{N}$, where $\pi : \widetilde{M} \rightarrow M$ is a smooth finite cover of $M$ and $w : M \rightarrow \mathcal{N}$ is a proper embedding of class $C^q$, admits a non-trivial Jacobi field. Notice that, as in Remark \ref{rmk:fund}, $\mathcal{B}^{q+1}(M)=\mathcal{B}^{q}(M)\cap\Gamma^{q+1}$ and $\mathcal{B}^{\infty}(M) = \mathcal{B}^{q}(M)\cap \Gamma^{\infty}$ for all integers $q\geq 3$.\\
  	\indent We first consider the case of finite $q \geq 3$. \\
  	
  	\noindent \textbf{Claim:} For each integer $q\geq 3$, there exists a sequence $\{U^{q}_{i}\}_{i\in \mathbb{N}}$ of subsets of $\Gamma^{q}$ such that:
  	\begin{itemize}
  		\item[a)] $U^{q}_{i}$ is an open, dense subset of $\Gamma^{q}$ for all $q\geq 3$ and $i\geq 0$.
  		\item[b)] $U^{q+1}_{i} = U^{q}_{i}\cap \Gamma^{q+1}$ for all $q\geq 3$ and $i\geq 0$.
  		\item[c)] $\mathcal{B}^{q}(M) \supset \bigcap_{i=1}^{\infty}{U^{q}_{i}}$.\\
  	\end{itemize}
  	
  	First of all, let us observe that by Theorem \ref{thm:struct}, $\pi^{q}$ is a $C^{1}$ Fredholm map of Fredholm index zero and, in particular, it is locally proper. Furthermore, if we denote by $\mathcal{R}^q\subset\mathcal{S}^q$ the set of regular points (i.e., points where the differential of $\pi^q$ is surjective) and its complement, the set of critical points, by $\mathcal{C}^q$, we get at once that $\mathcal{R}^q$ is open in $\mathcal{S}^q$ (since the Fredholm index of $\pi^q$ is zero, each regular point is actually a point where the derivative of $\pi^q$ is an isomorphism, so that the Inverse Function Theorem applies).  \\
  	\indent Let $\mathcal{C}^{q}_{p}$ be the set of all pairs $(\gamma,[\omega])\in \mathcal{S}^q$ such that there exists a smooth $p$-sheeted covering $\pi : \widetilde{M} \rightarrow M$ such that $w\circ \pi$ admits a Jacobi field (cf. \cite{Whi15}, Definition 2.5). Again by elliptic regularity, we have $\mathcal{C}^{q+1}_{p} = \mathcal{C}^{q}_{p}\cap \mathcal{S}^{q+1}$. Furthermore, using the characterization provided by Theorem \ref{thm:struct}, namely the fact that $(\gamma,[w]) \in \mathcal{S}^{q}$ is a critical point
  	of $\pi^q$ if and only if $w(M)$ admits a non-trivial Jacobi field, and the digression on the regularity of these Jacobi fields given in Remark \ref{rmk:regul} one can then check that $\mathcal{C}^{q+1} = \mathcal{C}^{q} \cap \mathcal{S}^{q+1}$.
  	Notice also that
  	\begin{equation*}
  	\mathcal{B}^{q}(M) =\bigcap_{p=1}^{\infty} (\Gamma^{q} \setminus \pi^q(\mathcal{C}^{q}_{p}))=\Gamma^q\setminus \bigcup_{p=1}^{\infty} \pi^q(\mathcal{C}^q_p)\supset \Gamma^q\setminus \left[\pi^q(\mathcal{C}^q)\cup \bigcup_{p=1}^{\infty}\pi^q(\mathcal{C}^{q}_p\setminus\mathcal{C}^q)\right].
  	\end{equation*}
  	It follows from the proof of Lemma 2.6 of \cite{Whi15} that, for each positive integer $p$, the set $\mathcal{C}^q_p$ is closed and $\overline{\mathcal{S}^q\setminus \mathcal{C}_p^{q}}\supset\mathcal{S}^q\setminus\mathcal{C}^q$ or, equivalently, $\textrm{Int}(\mathcal{C}^{q}_p)\subset\mathcal{C}^q$. That proof is based on a local deformation of the metric on the interior of $\mathcal{N}$, which can be carried out in our setting to the same effect with no significant changes in the argument. Since $\mathcal{S}^q$ is separable (see the statement of Theorem \ref{thm:struct}), $\pi^q$ restricted to $\mathcal{R}^q$ is locally a $C^1$-diffeomorphism, and $\pi^q$ is locally proper\footnote{For the construction presented here, the reader should recall that a proper map whose target is a first-countable Hausdorff space is in fact closed.}, we can choose two sequences of sets:
  	\begin{itemize}
  		\item{$\left\{V^{q,p}_k\right\}, \ k\geq 1$ such that $\cup_{k\geq 1}V^{q,p}_k$ covers $\mathcal{C}^q_p\setminus\mathcal{C}^q\subset\mathcal{R}^q$, and the restriction of $\pi^q$ to each $\overline{V^{q,p}_k}$ is an homeomorphism onto its image;}
  		\item{$\left\{F^q_l\right\} \ l\geq 1$ such that $\mathcal{C}^q=\cup_{l\geq 1}F^q_l$, and each $F^q_l$ is a closed subset of $\mathcal{S}^q$ with $\pi^q(F^q_l)$ also closed in $\Gamma^q$.}	
  	\end{itemize}	Most importantly, in both cases these sets can be chosen so that the sets defined by $U^{q,p}_k=\Gamma^{q}\setminus \pi^{q}(\mathcal{C}^{q}_p\cap \overline{V^{q,p}_{k}})$ and $U^q_l:=\Gamma^{q}\setminus \pi^{q}(F^{q}_{l})$ satisfy the nested property b) described in the Claim (this relies on Remark \ref{rmkdiag}, for one can define these sets inductively starting at level $q=q_0=3$). 
  	Furthermore, each set $\pi^q(F^q_l)$ and $\pi^q(\mathcal{C}^{q}_p\cap\overline{V^{q,p}_k})$ is patently closed and we also claim with empty interior. Indeed, in the former case it is enough to observe that $\pi^q(F^q_l)\subset\pi^q(\mathcal{C}^q)$, and invoke the Sard-Smale lemma (Theorem \ref{thm:Sma}) which ensures that $\pi^q(\mathcal{C}^q)$ is meagre. In the latter case, one argues as follows: $\mathcal{C}^{q}_p\cap \overline{V^{q,p}_k} \subset \mathcal{C}^{q}_p\setminus \mathcal{C}^q$ by the very choice of $V^{q,p}_k$, also $ \ \textrm{Int}(\mathcal{C}^{q}_{p}\setminus\mathcal{C}^q)\subset\textrm{Int}(\mathcal{C}^q_p)\setminus\mathcal{C}^q=\emptyset$ as seen above, hence each $\mathcal{C}^{q}_p\cap \overline{V^{q,p}_k}$ has empty interior and thus its image under the homeomorphisms $\pi^{q}: \overline{V^{q,p}_{k}} \rightarrow \pi^{q}(\overline{V^{q,p}_k})$ is also a closed subset of $\Gamma^{q}$ with empty interior.
  	
  	As far as property c) is concerned, just observe that by definition of the sets in question one has
  	\[
  	\pi^q(\mathcal{C}^q)\cup\pi^q(\mathcal{C}^{q}_p\setminus\mathcal{C}^q)\subset \bigcup_{l\geq 1} \pi^q(F^q_l)\cup \bigcup_{k\geq 1} \pi^q(\overline{V^{q,p}_k})
  	\]
  	
  	Hereby, after relabelling the sets $U^{q,p}_{k},U^{q}_{l}$ defined in the above argument (so to replace the indices $k,l,p$ by the sole index $i$), the proof of the claim is complete.\\
  	\indent Now we use the claim to derive the conclusion for $q=\infty$. First, we observe the following general fact:

  	\begin{lem} \label{lemnested}
  		Let $\{U^q\}_{q\geq q_0}$ be a sequence of open dense subsets of $\Gamma^{q}$ such that $U^{q+1}=U^{q}\cap \Gamma^{q+1}$ for all $q \geq q_0$. Then, 
  		\begin{equation*}
  		U^{\infty}:= \cap_{q\geq q_0} (U^{q}\cap \Gamma^{\infty})
  		\end{equation*}
  		\noindent is an open and dense subset of $\Gamma^{\infty}$ (with respect to the $C^\infty$ topology). 
  	\end{lem}
  	
  	\begin{proof}
  		The nested property, $U^{q+1}=U^{q}\cap \Gamma^{q+1}$, guarantees that, in fact, $U^{\infty} = U^{q}\cap \Gamma^{\infty}$ for each given $q\geq q_0 $. In particular, $U^{\infty}$ is clearly an open subset of $\Gamma^{\infty}$. Recall that the sets $B^{q}_\varepsilon(\gamma_0)\cap \Gamma^{\infty}$, where $q\geq q_0$, $\varepsilon > 0$ and $\gamma_0\in \Gamma^{\infty}$ are arbitrary and $B^{q}_{\varepsilon}(\gamma_0)$ denotes the open ball in $\Gamma^{q}$ centered at $\gamma_0$ with radius $\varepsilon$, form a basis for the topology of $\Gamma^{\infty}$ (cf. Remark \ref{rmk:cqtop}). Since $U^{q}$ is open and dense in $\Gamma^{q}$, $U^{q}\cap B^{q}_{\varepsilon}(\gamma_0)$ is a non-empty open subset of $\Gamma^{q}$. Using the standard approximation result that $\Gamma^{\infty}$ is dense in $\Gamma^{q}$ (with respect to the $C^q$-topology), we conclude that $U^{\infty} \cap B^{q}_{\varepsilon}(\gamma_0) = \Gamma^{\infty}\cap U^{q}\cap B^{q}_{\varepsilon}(\gamma_0)$ is non-empty. This shows that $U^{\infty}$ is dense in $\Gamma^{\infty}$, and concludes the proof.
  	\end{proof}
  	
  	\indent The sets $U^{q}_{i}$ defined in the Claim satisfy the hypotheses of the Lemma \ref{lemnested}. Thus, the sets $U^{\infty}_{i} = \cap_{q\geq 3} (U^{q}_{i}\cap \Gamma^{\infty})$ are open and dense subsets of $\Gamma^{\infty}$ such that
  	
  	\begin{align*}
  	\mathcal{B}^{\infty}(M)&=\bigcap_{q\geq q_0}(\mathcal{B}^q(M)\cap\Gamma^{\infty})\supseteq \bigcap_{q\geq q_0}\left(\cap_{i\geq 0}U^q_i \cap \Gamma^{\infty}\right)\\
  	&=\bigcap_{i\geq 0}\left(\cap_{q\geq q_0}(U^q_i\cap \Gamma^\infty)\right)=\bigcap_{i\geq 0}U^{\infty}_i.
  	\end{align*}
  	It follows that $\mathcal{B}^{\infty}(M)$ is a comeagre subset of $\Gamma^{\infty}$. \\
  	\indent Since there are only countably many diffeomorphisms types of compact smooth manifolds with boundary (by virtue of the classical, aforementioned results \cite{Cai35, Whi40}), one has that $\mathcal{B}^{q}= \cap_{M} \mathcal{B}^{q}(M)$ is a comeagre subset of $\Gamma^{q}$ for all integers $q\geq 3$ and $q=\infty$, which establishes Theorem \ref{thm:bumpy}.
  \end{proof}

  \subsection{Proof of Theorem \ref{thm:struct}, the structure theorem}\label{subs:struct}
  
  The logical scheme we are about to follow resembles the one presented by B. White in \cite{Whi91}, so we will mostly focus on precisely describing the necessary changes that are needed here, to handle the case of elliptic functionals defined on manifolds with boundary. Since the main statements are fairly general and hence rather abstract, we have decided to first outline the application we aim at, so that the reader can keep in mind that special case throughout our discussion.
  
  \
  
  Given $\mathcal{N}^{n+1}$ a smooth, compact manifold with non-empty boundary we endow it with a fixed, smooth, background metric $\gamma_{\ast}$ with the property that $\partial\mathcal{N}$ is totally geodesic in $\mathcal{N}$. The existence of such a metric is trivial, since it is enough to interpolate (by means of a smooth cutoff function) the product metric on a collar neighborhood of $\partial\mathcal{N}$ with any smooth metric on $\mathcal{N}$.   
  Throughout this section of the article, all differential operators are always tacitly understood with respect to the metric $\gamma_{\ast}$, unless it is explicitly stated otherwise. The same remark also applies to all Hausdorff measures coming into play, to be considered induced by $\gamma_{\ast}$. \\

  \indent For a smooth (namely: $C^{\infty}$) proper embedding $w:M\to\mathcal{N}$ such that $w(M)$ meets $\partial\mathcal{N}$ orthogonally along $\partial w(M)$ (with respect to the metric $\gamma_{\ast}$) we wish to identify embeddings $w_1:M\to\mathcal{N}$ whose image is geometrically close to $w(M)$ with suitable sections of the normal vector bundle of $w(M)\subset\mathcal{N}$. To do that and to formalize this idea, we simply proceed as follows.
  
  Let $V$ denote the rank one smooth vector bundle over $M$ obtained as pull-back of the normal bundle of $w(M)\subset\mathcal{N}$ and let $V^r:\left\{(x,v)\in V \ | \ |v|<r|\right\}$. It is well-known that for $r>0$ small enough, the exponential map of $(\mathcal{N},\gamma_{\ast})$ provides a diffeomorphism $E:V^r\to\mathcal{U}$ for some open neighborhood $\mathcal{U}$ of $w(M)$ in $\mathcal{N}$ (recall that $\partial\mathcal{N}$ is totally geodesic with respect to the metric $\gamma_{\ast}$). In turn, this induces a natural identification between the open sets
  \[
  \mathcal{B}([w],\delta_G):=\left\{[w_1]\in C^{j,\alpha}(M,\mathcal{N})/\simeq \ | \ d([w],[w_1])<\delta_G)\right\} 
  \]
  and
  \[
  \mathcal{B}^V(0,\delta_L):=\left\{u_1\in C^{j,\alpha}(M,V) \ | \ \ \|u_1\|_{C^{j,\alpha}(M,V)}<\delta_{L} \right\}
  \]
  for suitably chosen, small $\delta_G, \delta_L >0$. More explicitly: given any $u_1\in \mathcal{B}^V(0,\delta_L)$ the equivalence class of the composite map $[E\circ u_1]$ belongs to $\mathcal{B}([w],\delta_G)$ and conversely given any equivalence class $[w_1]\in 	\mathcal{B}([w],\delta_G)$ there exists a unique $u_1\in \mathcal{B}^V(0,\delta_L)$ such that the representation $[w_1]=[E\circ u_1]$ holds true.\\

  \indent  In describing the local structure of the moduli spaces $\mathcal{S}^{q,j,\alpha}$ we will employ such identification in the following fashion. Given any $w_0\in C^{j,\alpha}(M,\mathcal{N})$ that is $\gamma_0$-stationary (namely: that is a free boundary minimal hypersurface in metric $\gamma_0$), standard results in differential topology ensure the existence of a map $w'\in C^{\infty}(M,\mathcal{N})$ that is as close as we wish to $w_0$ in the topology of $C^{j}(M,\mathcal{N})$.
  We can now modify such map in a tubular neighborhood of $\partial M$ and construct $w\in C^{\infty}(M,\mathcal{N})$ that is a proper embedding meeting $\partial\mathcal{N}$ orthogonally (with respect to $\gamma_{\ast}$), coinciding with $w'$ at boundary points and for which we can write $w_0=E\circ u_0$ for some $u_0\in C^{j,\alpha}(M,V)$ whose $C^0$ norm can be made as small as needed. At that stage, the identification above can be applied to describe all maps that are close to $w_0$ in $C^{j,\alpha}(M,\mathcal{N})/\simeq$. In particular, notice that for $u\in C^{j,\alpha}(M,V)$ the area formula (Cf. e. g. Section 8 of \cite{Sim83}) allows to write the area of $E\circ u$ in metric $\gamma$ in the form $\int_M a_{\gamma}(x,u(x),\nabla u(x))\,d\mathscr{H}^n$ where $a_{\gamma}$ is a parametric integrand satisfying suitable ellipticity axioms, as we are about to see. \\

  \indent In the statements below, let $j\geq 2$ be an integer, $\alpha\in (0,1)$ and $q\geq j+1$ also an integer. 
  
  \begin{rmk}\label{rem:rankone}
  	Given $(M^n,g)$ a smooth Riemannian manifold, throughout this section we let $V$ denote a rank one vector bundle over $M$ and $\nabla$ denote a metric connection on $V$. The product on $V$ shall simply be denoted by $\cdot$ and the musical isomorphisms $\sharp$ and $\flat$ are those induced by such metric. Also, we shall use the notation $V_M$ for the bundle $T^{\ast}M\otimes V\cong End(TM,V)$.
  \end{rmk}

  \begin{prop}(Cf. Theorem 1.1 in \cite{Whi91})\label{prop:firstvar}
  	Let $M^{n}$ be a smooth, compact manifold with boundary, let $\nabla$ be a smooth metric connection on a rank one smooth bundle $V$ and let $F:\Gamma\to \Phi$ be a smooth map where
  	\begin{itemize}
  		\item{$\Gamma$ is an open subset of a (linear) Banach space;}
  		\item{
  			$\Phi =\left\{f\in C^q(M\times V\times V_M,\R), \newline
  			\ D_3 f\in C^q(M\times V\times V_M, {V_M}^{\ast})\right\}.$
  		}
  		
  	\end{itemize}
  	If we write $f_{\gamma}$ in lieu of $F(\gamma)$ and set for $u\in C^{j,\alpha}(M,V)$
  	\[
  	F(\gamma,u)=\int_{M}f_{\gamma}(x,u,\nabla u(x))\,d\mathscr{H}^n 
  	\]
  	then
  	
  	\begin{align*}
  	\left.\frac{d}{dt}\right|_{t=0}F(\gamma,u+tv)&=\int_M \left(D_2 f_{\gamma}(x,u,\nabla u)[v]-(\textrm{div} \ D_3 f_{\gamma}(x,u,\nabla u))[v]\right)\,d\mathscr{H}^n \\
  	&+\int_{\partial M}(D_{3}f_{\gamma}(x,u,\nabla u)[\nu^{\flat}\otimes v]\,d\mathscr{H}^{n-1}
  	\end{align*}
  	
  	where $\nu$ is the outward-pointing conormal of $\partial M$ in $M$. Furthermore:
  	\begin{enumerate}
  		\item{the first variation map $(H,\Theta): \Gamma\times C^{j,\alpha}(M,V)\to C^{j-2,\alpha}(M,V)\times C^{j-1,\alpha}(\partial M,V)$ given by
  			\[
  			\ \ \ \ \ \ \ \	H^{\flat}(\gamma, u)=-D_2 f_{\gamma}(x,u,\nabla u)+\textrm{div} \ D_3 f_{\gamma}(x,u,\nabla u), \ \ \Theta^{\flat}(\gamma, u)=(D_{3}f_{\gamma}(x,u,\nabla u)[\nu^{\flat}\otimes\cdot])
  			\]
  			is of class $C^{q-j}$.}
  		\item{if $u\in C^{j,\alpha}(M,V)$ satisfies 
  			\[
  			\begin{cases}
  			H(\gamma,u)=0 \\
  			\Theta(\gamma,u)=0
  			\end{cases}
  			\] then $u\in C^{q,\beta}(M,V)$ for any $\beta<1$ provided 
  			\begin{equation}\label{eq:ellipt}
  			\inf_{x\in M} \inf_{\zeta\in T^{\ast}_x M\otimes V_x} D_{33}f_{\gamma}(x,u,\nabla u)[\zeta,\zeta]\geq \varepsilon |\zeta|^2
  			\end{equation}
  			for some $\varepsilon>0$.
  		}
  	\end{enumerate}
  \end{prop}
  
  \begin{rmk}(Identifications).
  	In the statement above we have tacitly exploited the isomorphism $(T^{\ast}M\otimes V)^{\ast}\simeq TM\otimes V^{\ast}$. Also, if $\tau_1,\ldots,\tau_n$ is a local orthonormal frame for $TM$ observe that $\textrm{div} \ D_3 f_{\gamma}(x,u,\nabla u))[v]=\sum_{j=1}^n\nabla_{\tau_j} (D_3 f_{\gamma}(x,u,\nabla u))[\tau_j^{\flat}\otimes v]$.
  \end{rmk}	
  
  \begin{rmk}
  	Let us explicitly observe that the bilinear form $D_{33}f_{\gamma}(x,u,\nabla u)[\cdot,\cdot]$ is symmetric, for indeed this just follows by computing the left-hand side and the right-hand side of the identity
  	\[
  	\left.\frac{\partial^2 }{\partial s \partial t}\right|_{s=t=0} f_{\gamma}(x,u(x),\nabla u(x)+s\zeta_1+t\zeta_2)=\left.\frac{\partial^2 }{\partial s \partial t}\right|_{s=t=0} f_{\gamma}(x,u(x),\nabla u(x)+t\zeta_1+s\zeta_2).
  	\]
  \end{rmk}

  Let us now get back to the proposition in question.

  \begin{proof}
  	The proof of the first assertion and of part (1) amounts to routine checks, so we omit the details. As far as the part (2) is concerned, let us start by noticing that (under our hypotheses) $u$ solves a \textsl{linear} inhomogeneous uniformly elliptic equation, as is seen by simply expanding the divergence term in the expression for the functional $H$. In particular, the ellipticity constant equals the coercivity constant of the symmetric bilinear form $D_{33}f_{\gamma}(x,u,\nabla u)[\cdot,\cdot]$ and thus our assumption \eqref{eq:ellipt} ensures uniform ellipticity. In order to apply global Schauder estimates (i.e. estimates up to the boundary, as per Theorem 6.30 in \cite{gt}) we also need to make sure that $u$ satisfies a linear oblique boundary condition. That is not necessarily the case for general $f_{\gamma}$, but there is a direct way of bypassing the obstacle. Indeed, since we have assumed throughout our discussion that $j\geq 2$ we already know that $u\in C^{2,\alpha}(M,V)$ and thus we can simply check that, said $\left\{\tau_{j}\right\}_{j=1,\ldots, n}$ a local orthonormal frame on an open set $\Omega\subset M$, $\nabla_{\tau_j} u \in C^{q-1,\beta}(\Omega,V)$ for any $\beta<1$: now, it is clear that $v:=\nabla_{\tau_j} u$ also solves a linear elliptic equation as above, plus (by differentiating $D_3 f_{\gamma}(x,u,\nabla u)[\nu^{\flat}\otimes\cdot]=0$) it satisfies a linear boundary condition of the form
  	\[
  	D_{33}f_{\gamma}(x,u,\nabla u)[\nabla v,\nu^{\flat}\otimes \cdot]=-D_{13}f_{\gamma}(x,u,\nabla u)[\tau_{j},\nu^{\flat}\otimes\cdot]-D_{23}f_{\gamma}(x,u,\nabla u)[v,\nu^{\flat}\otimes\cdot]
  	\]
  	and thus the transversality condition that one needs is also easily checked to be implied by \eqref{eq:ellipt}. Indeed if $v_{\ast}$ provides a local trivialization of $V$ over $\Omega$ (possibly by considering a smaller domain, without renaming) there is a natural isomorphism $C^{j,\alpha}(M,V)\simeq C^{j,\alpha}(M,\R)$ that associates a real-valued map $\hat{u}$ to a map $u$, and a functional $\hat{f}_{\gamma}$ to $f_{\gamma}$: in this framework one can diagonalize the symmetric quadratic form $D_{33}\hat{f}_{\gamma}(x,\hat{u},\nabla \hat{u})=\textrm{diag}(\lambda_1,\ldots,\lambda_n)$, the ellipticity condition being just the requirement $\min_{i=1,\ldots, n} \lambda_i >0$ and the oblique derivative condition being that $D_{33}\hat{f}_{\gamma}(x,\hat{u},\nabla\hat{u})[\nabla \hat{v},\nu^{\flat}\otimes\cdot]=\sum_{i=1}^n\lambda_i \nu_i \nabla_{\tau_i}\hat{v}\geq c\frac{\partial \hat{v}}{\partial\nu}$, for some $c>0$ (here ${\tau_i}$ is specified to be a local, orthonormal, diagonalizing frame). 
  	Therefore, considering the restriction of $u$ to finitely many such domains (by compactness of $M$), the corresponding local trivializations of the bundle $V$ and invoking linear Schauder estimates as provided by Lemma 6.29 in \cite{gt} on each of these domains we obtain the desired conclusion since
  	it is enough to notice that if $u\in C^{j,\alpha}(M,V)$ then all coefficients as well as the inhomogeneous term of the equation for $v$ are in $C^{j-2,\alpha}$, which allows to prove that $v\in C^{q-2,\beta}(M,V)$ for any $\beta<1$ and then, as a final stage, we feed this information again in the equation for $v$ to gain that in fact $v\in C^{q-1,\beta}(M,V)$ as was claimed.   
  \end{proof}

  \begin{rmk}\label{rem:ellipt}
  	The fact that the area integrand $a_{\gamma}(x,z,p)$ satisfies the ellipticity condition \eqref{eq:ellipt} follows from the corresponding assertion for the integrand for Cartesian graphs in Euclidean spaces together with a standard approximation argument. Indeed, recall that if $\Omega\subset\R^n$ is a regular domain and $u\in C^{1}(\Omega,\R)$ then the function in question takes the form $a_{\delta}(x,z,p)=\sqrt{1+|p|^2}$ for which one has
  	\[
  	D_{33}a_{\delta}(x,z,p)=\frac{\delta_{ij}(1+|p|^2-p_i p_j)}{(1+|p|^2)^{3/2}}
  	\]	
  	so that \eqref{eq:ellipt} is satisfied with constant $\varepsilon=1$. Now, assume instead to have $w\in C^{\infty}(M,\mathcal{N})$ and $w_0=E\circ u_0$ for some $u_0\in C^{j,\alpha}(M,V)$ whose $C^0$ norm is small compared to the injectivity radius of $(\mathcal{N},\gamma_{\ast})$: given any $\rho$ one can find finitely many balls $B_1,\ldots, B_{|I|}$ with $B_i=B_{r_i}(p_i)$  covering the image $w(M)$ and such that (in the associated geodesic normal coordinates centered at $p_1,\ldots, p_{|I|}$ respectively) one has $\|g_{ij}-\delta_{ij}\|_{C^2(\overline{B}_i)}<C\rho^2$ and the restriction $w|_{B_i}$ is an exponential graph over $T_{p_i}w(M)$ whose defining function has $C^2$-norm less than $C\rho^2$. These facts suffice to check the pointwise condition \eqref{eq:ellipt} provided we only choose $\rho>0$ small enough. 
  \end{rmk}	
  
  We now turn our attention to the the linearizations of the maps $H$ and $\Theta$ defined above. 	
  
  \begin{prop}\label{prop:jacobi}
  	The linearization of the first variation operators are:
  	\[
  	L_{H}: C^{j,\alpha}(M,V)\to C^{j-2,\alpha}(M,V), \  L_H(\gamma,u)[v]=\left.\frac{d}{dt}\right|_{t=0}H(\gamma,u+tv) 
  	\]
  	given by
  	\begin{align}\label{eq:jh}
  	L_H(\gamma,u)[v] &=-(D_{22}f_{\gamma}(x,u,\nabla u)[\cdot, v]+D_{23}f_{\gamma}(x,u,\nabla u)[\cdot, \nabla v])^{\sharp} \\
  	&+div\left(D_{23}f_{\gamma}(x,u,\nabla u)[v,\cdot]+D_{33}f_{\gamma}(x,u,\nabla u)[\nabla v,\cdot] \right)^{\sharp}
  	\end{align}
  	and 
  	\[
  	L_{\Theta}: C^{j,\alpha}(M,V)\to C^{j-1,\alpha}(\partial M,V), \ L_\Theta(\gamma,u)[v]=\left.\frac{d}{dt}\right|_{t=0}\Theta(\gamma,u+tv)
  	\]
  	given by
  	\begin{equation}\label{eq:jt}
  	L_{\Theta}(\gamma,u)[v]=(D_{23}f_{\gamma}(x,u,\nabla u)[v,\nu^{\flat}\otimes\cdot]+D_{33}f_{\gamma}(x,u,\nabla u)[\nabla v,\nu^{\flat}\otimes\cdot])^{\sharp}. 
  	\end{equation}
  	Moreover, if the ellipticity condition \eqref{eq:ellipt} holds then:
  	\begin{enumerate}
  		\item{there exists an orthonormal basis of $L^2(M,V)$ consisting of Robin eigenfunctions for $(L_{H},L_{\Theta})$, namely a sequence $\left\{\phi_k\right\}_{k\geq 1}$ satisfying
  			\[
  			\begin{cases}
  			L_{H}{\phi_k}=-\lambda_k \phi_k \\
  			L_{\Theta}{\phi_k}=0
  			\end{cases}
  			\]
  			where $\lambda_1<\lambda_2\leq\ldots\lambda_k\to+\infty$;}
  		\item{the operator $L=(L_{H},L_{\Theta})(\gamma,u):C^{j,\alpha}(M,V)\to C^{j-2,\alpha}(M,V)\times C^{j-1,\alpha}(\partial M,V)$ is Fredholm of index 0;}
  		\item{denoted by $\epsilon:\partial M\to M$ the canonical inclusion, and by $(\cdot,\cdot)$ the product coupling on $L^2(M,V)\times L^2(\partial M,V)$ one has
  			\[
  			(L(\gamma,u)[v],(w,w\circ\epsilon))=(L(\gamma,u)[w],(v,v\circ\epsilon)) \ \forall v,w \in C^{j,\alpha}(M,V).
  			\]}	
  	\end{enumerate}	
  \end{prop}
  
  \begin{proof} Part (1) is standard and part (3) follows at once via integration by parts, so let us discuss part (2).
  	First of all, Schauder estimates (up to the boundary) for oblique derivative problems (Cf. Theorem 6.30 in \cite{gt}, but consider also the discussion above about local trivializations of $V$) imply in our case that 
  	\[
  	\|\phi\|_{C^{j,\alpha}(M,V)}\leq C\left(\|\phi\|_{C^0(M,V)}+\|L_H \phi\|_{C^{j-2,\alpha}(M,V)}+\|L_{\Theta}\phi\|_{C^{j-1,\alpha}(\partial M,V)} \right)
  	\]
  	and hence the fact that $L$ is a Fredholm operator follows by general arguments in Functional Analysis. More specifically, that the image of $L$ is closed follows e. g. from Proposition 3.1 in Chapter 5 of \cite{Tay81}, that the kernel $K$ is finite dimensional follows from the fact that (by virtue of the estimate above, and the Arzel\'a-Ascoli theorem) $K\cap C^{j,\alpha}(M,V)$ has a compact unit ball and all we need to check now is that indeed $\textrm{dim}(K)=\textrm{codim}(R)$ (where $R\subset C^{j-2,\alpha}(M,V)\times C^{j-1,\alpha}(\partial M,V)$ denotes the image of $L$), which means that the Fredholm index of $L$ equals zero. Given $(\xi, \theta)\in C^{j-2,\alpha}(M,V) \times C^{j-1,\alpha}(\partial M,V)$, we want to determine conditions for the solvability of the problem $L\phi=(\xi,\theta)$ for $\phi\in C^{j,\alpha}(M,V)$.
  	Let then $\psi\in C^{j,\alpha}(M,V)$ be a function satisfying $L_{\Theta}\psi=\theta$ (for instance, Theorem 6.31 in \cite{gt} ensures the existence of an \textsl{harmonic} such function). So, by linearity it is clear that one can solve our problem if and only if the system
  	\begin{equation}\label{eq:ellsys}
  	\begin{cases}
  	L_H \phi= \xi-L_H \psi \\
  	L_{\Theta} \phi =0
  	\end{cases}
  	\end{equation}
  	is solvable for $\phi\in C^{j,\alpha}(M,V)$. Now, elementary arguments ensure that this happens if and only if $\xi-L_H \psi \in K^{\perp}$ and the conclusion follows at once.
  \end{proof}

  Before presenting the next proposition, which is reminescent of Theorem 1.2 in \cite{Whi91}, we need to introduce some notation.
  Throughout this section, we let 
  \[
  X=C^{j,\alpha}(M,V), \ \ \check{X}=\left\{(\varphi,\varphi\circ\epsilon), \ \varphi\in X \right\}
  \]
  \[
  Y=Y_1\times Y_2, \ Y_1=C^{j-2,\alpha}(M,V) \ Y_2=C^{j-1,\alpha}(\partial M,V)
  \]
  so that we obviously have the inclusion $\check{X}\subset Y\subset L^2(M)\times L^2(\partial M)$.
  For a point $(\gamma_0,u_0)$ satisfying $H(\gamma_0,u_0)=\Theta(\gamma_0,u_0)=0$, we further set 
  \[
  K=ker L(\gamma,u), \ \ \check{K}=\left\{(\varphi,\varphi\circ\epsilon), \ \varphi\in K \right\}
  \]
  and said $\pi_K: L^2(M)\to K$ the orthogonal projector (with respect to the $L^2(M)$-structure) we introduce the associated operator:
  \[
  \pi^{ext}_{K}: X\to \check{K}, \ \ \pi^{ext}_{K}(\varphi)=(\pi_K(\varphi),\pi_K(\varphi)\circ\epsilon).
  \] 
  Notice that $\check{K}$ is a finite dimensional subspace of $L^2(M,V)\times L^2(\partial M,V)$, hence it is trivially closed and has its own orthogonal projector $\pi_{\check{K}}:\check{X}\to \check{K}$.

  \begin{prop}\label{prop:submersion}
  	In the setting above, and under the hypotheses of Proposition \ref{prop:firstvar} and Proposition \ref{prop:jacobi}, assume that for every nonzero $\kappa\in K$ there exists a differentiable curve $\gamma:(-1,1)\to\Gamma$ with $\gamma(0)=\gamma_0$ and such that
  	\begin{equation}\label{eq:submhp}
  	\left.\frac{\partial}{\partial s}\right|_{s=0}((H(\gamma(s),u_0),\Theta(\gamma(s),u_0)),(\kappa,\kappa\circ\epsilon))\neq 0.
  	\end{equation}
  	Then the map $(H,\Theta): \Gamma\times X\to Y$ is a submersion near $(\gamma_0,u_0)$, so there exists a neighborhood $U_{\Gamma}\times U_{X}$ of such point such that
  	\[
  	\mathcal{M}=\left\{(\gamma,u)\in U_{\Gamma}\times U_{X} \ | \ \ H(\gamma,u)=0, \ \Theta(\gamma,u)=0 \right\}
  	\] 
  	is a $C^{q-j}$-Banach submanifold with tangent space $ker D(H,\Theta)(\gamma,u)$ at each $(\gamma,u)\in\mathcal{M}$. Furthermore, the restricted projection operator $\Pi_{|\mathcal{M}}:\mathcal{M}\to\Gamma$, where $\Pi:\Gamma\times X\to\Gamma$ is given by $\Pi(\gamma,u)=\gamma$, is a $C^{q-j}$-Fredholm map of index 0.
  \end{prop}
  
  \begin{proof}
  	For the first assertion, we wish to apply the well-known submersion criterion in the setting of Banach manifolds (cf. e. g. Prop. 2.3 in \cite{Lan99}), so we need to check surjectivity of $D(H,\Theta)$ at $(\gamma_0,u_0)$ and splitting of the domain (which is implied by the existence of a bounded linear projector onto $ker D(H,\Theta)(\gamma_0,u_0)\subset \Gamma\times X$). As far as surjectivity is concerned, let us first observe that $LX=\check{K}^{\perp}\cap Y$ since the inclusion  $LX\subseteq \check{K}^{\perp}\cap Y$ follows from property (3) in Proposition \ref{prop:jacobi} and the two linear spaces in question must have the same codimension because $L$ is of index 0 by virtue of property (2) again in Proposition \ref{prop:jacobi}. Now, for the sake of a contradiction assume that there exists $\kappa\in K$ such that, without loss of generality, $(k,k\circ\epsilon)\perp D_1(H,\Theta)(\gamma_0,u_0) X$ (with respect to $L^2(M,V)\times L^2(\partial M,V)$). But then, said $\eta=\gamma'(0)$ one would have
  	\begin{align*}
  	0 =(D_1 (H,\Theta)(\gamma_0,u_0)[\eta],(\kappa,\kappa\circ\epsilon))=&\int_M \gamma_{\ast}(\kappa, D_1 H(\gamma_0,u_0)[\eta])\,d\mathscr{H}^n\\
  	&+\int_{\partial M}\gamma_{\ast}(k, D_1\Theta(\gamma_0,u_0)[\eta])\,d\mathscr{H}^{n-1} \\
  	= &\left.\frac{\partial}{\partial s}\right|_{s=0}((H(\gamma(s),u_0),\Theta(\gamma(s),u_0)),(\kappa,\kappa\circ\epsilon))\neq 0
  	\end{align*}
  	where the last step relies on our assumption \eqref{eq:submhp}. Such contradiction completes the proof that $D(H,\Theta)(\gamma_0,u_0):\Gamma\times X\to Y$ is a surjective linear map. Notice that we have just shown that in fact 
  	\begin{equation}\label{eq:image}
  	ran(\pi_{\check{K}}\circ D_1 (H,\Theta)(\gamma_0,u_0))=\check{K}.
  	\end{equation}

  	Let us now consider the splitting issue. Based on the argument above, it is straightforward to obtain the characterization
  	\begin{equation}\label{eq:char}
  	ker D(H,\Theta)(\gamma_0,u_0)=\left\{(\eta,v), \ \eta\in ker (\pi_{\check{K}}\circ D_1 (H,\Theta)(\gamma_0,u_0)), \ Lv=-D_1 (H,\Theta)(\gamma_0, u_0)\eta \right\}.
  	\end{equation}
  	Because of \eqref{eq:image}, $K_\Gamma := ker (\pi_{\check{K}}\circ D_1 (H,\Theta)(\gamma_0,u_0))$ has codimension $dim K$ inside $\Gamma$ so (this being finite), we know by Hahn-Banach the existence of a Banach projector $\pi_{K_{\Gamma}}: \Gamma\to K_{\Gamma}$. That being said, we claim that a projector onto $ker D(H,\Theta)(\gamma_0,u_0)\subset \Gamma\times X$ is given by the map $\Psi: \Gamma\times X \to \Gamma\times X$ defined by
  	\begin{equation}\label{eq:defproj}
  	(\eta, v) \mapsto (\pi_{K_{\Gamma}}\eta, (\pi^{ext}_{K}+L)^{-1}(-D_1 (H,\Theta)\pi_{K_{\Gamma}}\eta+\pi^{ext}_{K}v)).
  	\end{equation}
  	To prove such claim, one needs to check the following facts:
  	\begin{enumerate}
  		\item{the map $\pi^{ext}_K+L: X\to Y$ is a linear isomorphism;}
  		\item{the operator $\Psi$ above is well-defined, linear and bounded;}
  		\item{$ran (\Psi)\subset ker D(H,\Theta)(\gamma_0,u_0)$}
  		\item{$\Psi=Id$ when restricted to $ker D(H,\Theta)(\gamma_0,u_0)$.}	
  	\end{enumerate}
  	and each of them is almost immediate to verify, so we leave the details to the reader. \\

  \indent	Lastly, we need to check that the restricted projector has Fredholm index equal to zero. On the one hand, it is trivial that $ker D(\Pi_{|\mathcal{M}})(\gamma_0,u_0)=\left\{0\right\}\times K$, as this simply follows by the identification of the tangent space to $\mathcal{M}$ at $(\gamma_0,u_0)$ with $ker D(H,\Theta)(\gamma_0,u_0)$. On the other hand, for the image of this map we can write (thanks to \eqref{eq:char})
  	\[
  	ran D(\Pi|_{\mathcal{M}})(\gamma_0,u_0)=\Pi(T_{(\gamma_0,u_0)}\mathcal{M})=\Pi(ker D(H,\Theta)(\gamma_0,u_0))=ker (\pi_{\check{K}}\circ D_1 (H,\Theta)(\gamma_0,u_0))
  	\]
  	and then the latter also has codimension equal to $dim K$ because of equation \eqref{eq:image}.
  	
  \end{proof}

  We can finally make use of this local description of the critical manifold to prove the \textsl{structure theorem}, stated above as Theorem \ref{thm:struct}.

  \begin{proof}
  	First of all, it follows from the discussion presented at the very beginning of this subsection that given any $(\gamma_0,[w_0])\in\mathcal{S}^{q,j,\alpha}$, one can actually find $\varepsilon>0$ small enough that the set
  	\[
  	\left\{(\gamma,[E\circ u]): \ H(\gamma,u)=\Theta(\gamma,u)=0, \ \|\gamma-\gamma_0\|_{C^q}+\|u-u_0\|_{C^{j,\alpha}}<\varepsilon \right\}
  	\]
  	is indeed an open neighborhood of the point in question inside $\mathcal{S}^{q,j,\alpha}$. Here $u_0$ is uniquely determined by the relation $w_0=E\circ u_0$.
  	That being said, we can construct a local $C^{q-j}$-chart for $\mathcal{S}^{q,j,\alpha}$ near $(\gamma_0,[w_0])$ by simply invoking Proposition \ref{prop:submersion} once we verify that the key assumption, condition \eqref{eq:submhp}, is satisfied. So, let $\kappa\in ker L(\gamma,u)$ (non zero): we need to construct a differentiable curve of metrics $s\mapsto\gamma(s)$ such that 
  	\[
  	\left.\frac{\partial^2}{\partial s \partial t}\right|_{s=t=0}\int_M a_{\gamma(s)}(x,u_0+t\kappa,\nabla (u_0+t\kappa))\neq 0
  	\]
  	which is clearly equivalent to the condition in question (we have set $a_{\gamma}$ to denote the non-parametric area functional in metric $\gamma$). The conformal deformation presented by B. White at page 179 of \cite{Whi91} serves the scope (the computation needed to check this can be followed with no changes at all, since no integration by parts is involved).
  	As a result, we can find a neighborhood $U_{\Gamma}\times U_{X}\subset \left\{\|\gamma-\gamma_0\|_{C^q}+\|u-u_0\|_{C^{j,\alpha}}<\varepsilon \right\}$ of $(\gamma_0, u_0)\in \Gamma\times C^{j,\alpha}(M,V)$ such that
  	\[
  	\mathcal{U}=\left\{(\gamma,u)\in U_{\Gamma}\times U_{X} \ | \ \ H(\gamma,u)=\Theta(\gamma,u)=0 \right\}
  	\] 
  	is a $C^{q-j}$-Banach submanifold in $\Gamma\times X$ and let $\varphi(\gamma,u)=(\gamma,u\circ E)$ be the associated local parametrization (which is patently a homeomorphism). All we need to check is that given any two such local parametrizations, say $(\mathcal{U},\varphi)$ and $(\mathcal{U}',\varphi')$ they are indeed $C^{q-j}$ compatible, in the standard sense that
  	(possibly by taking a smaller $\mathcal{U}$, if needed, but without renaming) the composition $(\varphi')^{-1}\circ \varphi:\mathcal{U}\to\mathcal{U}'$ is a $C^{q-j}$ map. Once again, this is handled at page 180 of \cite{Whi91} and no modifications are needed in our case. Lastly, the fact that $\mathcal{S}^{q,j,\alpha}$ is separable is obtained as follows. Recall that, by virtue of statement (2) in Proposition \ref{prop:firstvar} if $w\in C^{j,\alpha}(M,\mathcal{N})$ is $\gamma$-stationary then in fact $w\in C^q(M,\mathcal{N})$, hence let $\check{\mathcal{S}}^{q,j,\alpha}$ be just the set $\mathcal{S}^{q,j,\alpha}$ endowed with the topology it inherits as a subspace of the product $\Gamma\times (C^{q}(M,\mathcal{N})/\simeq)$, for $\simeq$ the usual equivalence relation (quotienting by diffeomorphisms of $M$). Clearly, $\mathcal{S}^{q,j,\alpha}$ has a coarser topology than $\check{\mathcal{S}}^{q,j,\alpha}$ so it is enough to show that $\check{\mathcal{S}}^{q,j,\alpha}$ is separable. On the other hand, the latter conclusion follows as a consequence of three elementary facts:
  	\begin{enumerate}
  		\item{$C^{q}(M,\mathcal{N})$ is a separable metric space, and so is $C^{q}(M,\mathcal{N})/\simeq$ since it is Hausdorff;}
  		\item{$\Gamma$ is a separable metric space, so (by the previous step) $\Gamma\times (C^{q}(M,\mathcal{N})/\simeq)$ is also a separable metric space;}
  		\item{$\check{\mathcal{S}}^{q,j,\alpha}$ is a separable metric space since it is a subset of a separable metric space.}
  	\end{enumerate}	            	 	 
  	Thereby the proof is complete. 
  \end{proof}

\appendix

\section{The second variation for smooth hypersurfaces}\label{app:2ndvar}

Our goal is to derive the second variation of the area functional for arbitrary hypersurfaces with boundary in $\mathcal{N}^{n+1}$. For convenience, we assume $\mathcal{N}$ is isometrically embedded in some Euclidean space of possibly large dimension $\mathbb{R}^d$. Let $\Sigma^n \emb \mathcal{N}^{n+1}$ be a smooth properly embedded hypersurface with boundary. We consider a two-parameter family of ambient variations $\Sigma(t,s)$ of $\Sigma=\Sigma(0,0)$ defined by $$\Sigma(t,s)=F_{t,s}(\Sigma) = \Psi_s \circ \Phi_t(\Sigma)$$ where $\Phi_t$ and $\Psi_s$ are the flows generated by compactly supported vector fields $X$ and $Z$ in $\mathbb{R}^d$, respectively. We assume both $X$ and $Z$ define vector fields in $\mathfrak{X}_\de$ when restricted to $\mathcal{N}$, therefore each $\Sigma(t,s)$ is a properly embedded hypersurface in $\mathcal{N}^{n+1}$. 

We have
$$\text{$\left.\pl{F_{t,s}}{s}(x)\right|_{t=s=0} = Z(x)$, $\left.\pl{F_{t,s}}{t}(x)\right|_{t=s=0} = X(x)$, and  $\left.\pl{}{t}\pl{}{s}F_{t,s}(x)\right|_{t=s=0} = D_X Z (x)$}.$$
Computing the second variation of volume as in \cite{Sim83}, Chapter 2, \S 9, we obtain
\begin{multline*}
\pl{}{t}\left.\pl{}{s}\h^n(\Sigma(t,s))\right|_{t=s=0} = \int_{\Sigma}\sum_{i=1}^{n}\la D_{\tau_i}D_{X}Z,\tau_i\ra +( \sum_{i=1}^{n}\la D_{\tau_i}X,\tau_i \ra )( \sum_{j=1}^{n}\la D_{\tau_j}Z,\tau_j \ra) \\
+ \sum_{i=1}^n  \la (D_{\tau_{i}} X)^{\bot_{\Sigma}}, (D_{\tau_{i}} Z)^{\bot_{\Sigma}}\ra - \sum_{i,j=1}^n \la \tau_{i}, D_{\tau_{j}} X\ra \la\tau_{j}, D_{\tau_{i}} Z\ra \id \h^n. 
\end{multline*}
In the above formula, the symbol $\perp_{\Sigma}$ denotes the component that is normal to $\Sigma$ in $\mathbb{R}^d$, $D$ denotes the Euclidean Levi-Civita connection, and $\{\tau_i\}_{i=1}^{n}$ is a choice of an orthonormal basis of $T_p\Sigma$ at each point $p$ in $\Sigma$ (the above sums are easily shown to be independent of such choice).

Since we assumed that $X$ and $Z$ belong to $\mathfrak{X}_\de$ when restricted to $\mathcal{N}$, decomposing the vectors fields $X,Z$ and $D_{X}Z$ into their components that are tangent and perpendicular to $\mathcal{N}$ it is possible to use the Gauss equation for the embedding of $\mathcal{N}$ in $\mathbb{R}^d$ to rewrite the above formula solely in terms of the geometry of the embedding $\Sigma \emb \mathcal{N}$:
\begin{multline*}
\pl{}{t}\left.\pl{}{s}\h^n(\Sigma(t,s))\right|_{t=s=0} = \int_{\Sigma} div_{\Sigma} (\nabla_{X}Z) + div_{\Sigma} (X) div_{\Sigma}(Z) + \sum_{i=1}^{n} \la (D_{e_i}X)^{\bot},(D_{e_i}Z)^{\bot}\ra \\
- \sum_{i=1}^n  R_{\mathcal{N}}(X,\tau_i,Z,\tau_i) -\sum_{i,j=1}^n \la \tau_{i}, \nabla_{\tau_{j}} X\ra \la \tau_{j}, \nabla_{\tau_{i}} Z)\ra \id \h^n,
\end{multline*}
where $\nabla$ and $R_\mathcal{N}$ are the Levi-Civita connection and Riemann tensor of $\mathcal{N}$ respectively, $\bot$ denotes the component that is normal to $\Sigma$ \textit{and} tangent to $\mathcal{N}$, and, given any vector field $Y$ that is tangent to $\mathcal{N}$, we write $div_\Sigma(Y) = \sum_{i=1}^{n} \la \nabla_{\tau_i}Y,\tau_i \ra$.

In the sequel of this Appendix, the Levi-Civita connection of $\Sigma$ will be denoted by $\D^{\Sigma}$ and  $\D^\bot$ will denote the connection of the normal bundle of $\Sigma$ in $\mathcal{N}$. The projections $^\bot$, $^\top$ are the normal and tangential projections for $\Sigma(t,s) \emb \mathcal{N}$ respectively so that
$$\text{$X=X^\top + X^\bot$ and $Z=Z^\top + Z^\bot$}$$
since they are both tangent to $\mathcal{N}$. According to that decomposition, we write
\begin{multline} \label{Aeq1}
\pl{}{t}\left.\pl{}{s}\h^n(\Sigma(t,s))\right|_{t=s=0}  = \int_{\Sigma} div_\Sigma (\D_{X}Z) + \mathcal{A}(X^{\bot},Z^{\bot}) \\ + \mathcal{B}(X^{\top},Z^{\top}) 
+ \mathcal{C}(X^{\top},Z^{\bot}) + \mathcal{C}(Z^{\top},X^{\bot}) \id  \h^{n},
\end{multline}
where, using the identity $div_{\Sigma}(Y) = div_{\Sigma} (Y^{\top}) - \la Y^{\top}, H\ra$ for $H$ the mean curvature vector of $\Sigma \emb \mathcal{N}$ and any vector field $Y$ tangent to $\mathcal{N}$, the terms $\mathcal{A}, \mathcal{B}$ and $\mathcal{C}$ are given by:
\begin{align}
\mathcal{A}(X^{\bot},Z^{\bot}) = & \la X^{\bot}, H\ra\la Z^{\bot},H \ra + \sum_{i=1}^n\la\Db_{\tau_i} (X^\bot) , \Db_{\tau_i} (Z^\bot)\ra \nonumber \\ 
& - \sum_{i=1}^{n} R_\mathcal{N} (X^\bot,\tau_i, Z^\bot,\tau_i) - \sum_{i,j=1}^n \la \tau_i,\nabla_{\tau_j} (X^\bot) \ra \la \tau_j, \nabla_{\tau_i} (Z^\bot) \rangle, \label{AeqA} \\ 
\mathcal{B}(X^{\top},Z^{\top}) = & div_{\Sigma}(X^{\top})div_{\Sigma}(Z^{\top}) + \sum_{i=1}^{n} \la A(X^{\top},\tau_i),A(Z^{\top},\tau_i))\ra \nonumber \\
& - \sum_{i=1}^{n} R_\mathcal{N} (X^\top,\tau_i, Z^\top,\tau_i) - \sum_{i,j=1}^{n} \la \tau_i,\nabla_{\tau_j} (X^\top) \ra \la \tau_j, \nabla_{\tau_i} (Z^\top) \ra,  \nonumber \\  
\mathcal{C}(X^{\top},Z^{\bot}) = & - \la Z^{\bot},H \ra div_{\Sigma}(X^\top) + \sum_{i=1}^n\la A(X^{\top},\tau_i), \nabla^{\bot}_{\tau_i}(Z^{\bot}) \ra  \nonumber \\ 
& - \sum_{i=1}^n R_{\mathcal{N}}(X^{\top},\tau_i,Z^{\bot},\tau_i) + \sum_{i=1}^n \la A(\nabla^{\Sigma}_{\tau_i}(X^{\top}),\tau_i),Z^{\bot} \ra. \nonumber
\end{align}

Using the Gauss equation for $\Sigma \emb \mathcal{N}$ (valid for all $U,V,W,Y$ tangent to $\Sigma$), namely
\begin{equation*}
R_{\Sigma}(U,V,W,Y) = R_{\mathcal{N}}(U,V,W,Y) + \la A(U,W),A(V,Y) \ra - \la A(U,Y),A(V,W) \ra,
\end{equation*}
we can write
\begin{align}
\mathcal{B}(X^{\top},Z^{\top}) = & div_{\Sigma}(X^{\top})div_{\Sigma}(Z^{\top}) + \la d (X^{\top})^{\flat} , d (Z^{\top})^{\flat} \ra  - \la \nabla^{\Sigma} (X^{\top}), \nabla^{\Sigma} (Z^{\top}) \ra \nonumber \\ & - Ric_{\Sigma}(X^{\top},Z^{\top}) + \la A(X^{\top},Z^{\top}),H \ra. \label{AeqB}
\end{align}
In the above formula, we use the standard musical notation relating vector fields to their dual one-forms, and $d$ denotes the exterior differential. Defining the one-form $\omega$ in $\Sigma$ by
\begin{equation} 
\omega(\xi) = \la A(Z^{\top},\xi),X^{\bot}\ra + \la A(X^{\top},\xi),Z^{\bot} \ra - \la Z^{\bot}, H \ra \la X^{\top}, \xi \ra - \la X^{\bot}, H \ra \la Z^{\top}, \xi \ra, \label{AeqO}
\end{equation}
for all vectors $\xi$ tangent to $\Sigma$, a straightforward computation using the Codazzi equation (valid for all for $U,V, W$ tangent to $\Sigma$ and $\eta$ normal to $\Sigma$), namely
\begin{equation*}
\la (\nabla_{U}A)(V,W),\eta \ra - \la (\nabla_{V}A)(U,W),\eta \ra = R_{\mathcal{N}}(U,V,\eta,W),
\end{equation*}
gives
\begin{equation}
\mathcal{C}(X^{\top},Z^{\bot}) + \mathcal{C}(Z^{\top},X^{\bot}) = div_{\Sigma} \omega + \la \nabla^{\bot}_{X^{\top}} (Z^{\bot}), H \ra + \la \nabla^{\bot}_{Z^{\top}} (X^{\bot}), H \ra. \label{AeqC}
\end{equation}
Finally,
\begin{align}
div_{\Sigma}(\nabla_{X}Z) = & div_{\Sigma}((\nabla_{X}Z)^{\top}) - \la \nabla_{X^{\bot}}(Z^{\bot}), H \ra - \la A(X^{\top},Z^{\top}), H \ra \nonumber \\ & - \la \nabla_{X^{\bot}}(Z^{\top}), H \ra - \la \nabla^{\bot}_{X^{\top}}(Z^{\bot}), H \ra. \label{AeqD}
\end{align}
Substituting equations (\ref{AeqA}),(\ref{AeqB}), (\ref{AeqC}) and (\ref{AeqD}) into \ref{Aeq1}, we obtain
\begin{align*}
\pl{}{t}\left.\pl{}{s}\h^n(\Sigma(t,s))\right|_{t=s=0}  = & \int_{\Sigma} \la\Db (X^\bot) , \Db (Z^\bot)\ra - Ric_\mathcal{N} (X^\bot, Z^\bot) - \la (X^\bot),(Z^\bot) \ra|A|^2 
\\
&  + \la X^{\bot}, H\ra\la Z^{\bot},H \ra - \la \nabla_{X^{\bot}} (Z^{\bot}), H \ra - \la [X^{\bot},Z^{\top}],H \ra \\
&   + div_{\Sigma}(X^{\top}) div_{\Sigma}(Z^{\top}) + \la d (X^{\top})^{\flat} , d (Z^{\top})^{\flat} \ra - \la \nabla^{\Sigma} X^{\top}, \nabla^{\Sigma} Z^{\top}\ra \\
& - Ric_{\Sigma}(X^{\top},Z^{\top}) + div_\Sigma ((\nabla_{X}Z)^{\top}) + div_{\Sigma}\omega  \id \h^{n}.
\end{align*}
The final manipulation we perform is to integrate by parts. Choosing a local orthonormal frame $\{\tau_i\}$ geodesic at a point $p$ in $\Sigma$, the following computation at $p$ (where we sum over repeated indices $i$ and $j$) yields:
\begin{eqnarray*}
	div_{\Sigma}(\nabla^{\Sigma}_{X^{\top}}(Z^{\top})) & = & \tau_i \la \nabla^{\Sigma}_{X^{\top}}(Z^{\top}),\tau_i \ra \\
	& = & \tau_i (\la \nabla^{\Sigma}_{\tau_j}(Z^{\top}),\tau_i \ra \la X^{\top},\tau_j \ra) \\
	& = & \la \nabla^{\Sigma}_{\tau_i}\nabla^{\Sigma}_{\tau_j}(Z^{\top}),\tau_i\ra\la X^{\top},\tau_j \ra + \la \nabla^{\Sigma}_{\tau_j}(Z^{\top}),\tau_i \ra \la \nabla^{\Sigma}_{\tau_i}(X^{\top}),\tau_j \ra \\
	& = & R_{\Sigma}(\tau_i,\tau_j,\tau_i,Z^{\top})\la X^{\top},\tau_j\ra + \la \nabla^{\Sigma}_{\tau_j} \nabla^{\Sigma}_{\tau_i}(Z^{\top}),\tau_i\ra \la X^{\top},\tau_j\ra + \la \nabla_{\tau_j}Z^{\top},\tau_i \ra\la \nabla_{\tau_i}X^{\top},\tau_j \ra  \\
	& = & Ric_{\Sigma}(X^{\top},Z^{\top}) + d(div_{\Sigma}(Z^{\top}))(X^{\top}) - \left(\la d (X^{\top})^{\flat} , d (Z^{\top})^{\flat} \ra - \la \nabla^{\Sigma} X^{\top}, \nabla^{\Sigma} Z^{\top}\ra\right). 
\end{eqnarray*} 
Therefore
\begin{multline*}
\int_{\Sigma}  div_{\Sigma}(X^{\top})div_{\Sigma}(Z^{\top}) + \la d (X^{\top})^{\flat} , d (Z^{\top})^{\flat} \ra  - \la \nabla^{\Sigma} X^{\top}, \nabla^{\Sigma} Z^{\top}\ra  - Ric_{\Sigma}(X^{\top},Z^{\top})  \id \h^{n}  \\
=  \int_{\Sigma} div_{\Sigma}(div_{\Sigma}(Z^{\top})(X^{\top})) - div_{\Sigma}(\nabla^{\Sigma}_{X^{\top}}(Z^{\top})) \id \h^{n} \\ = \int_{\partial \Sigma} div_{\Sigma}(Z^{\top})\la X^{\top}, \nu \ra -  \la \nabla_{X^{\top}}(Z^{\top}),\nu \ra \id \h^{n-1},
\end{multline*}
where $\nu$ denotes the unit outward conormal of $\Sigma$. Thus,
\begin{align} \label{AeqX}
\pl{}{t}\left.\pl{}{s}\h^n(\Sigma(t,s))\right|_{t=s=0}  
=&\int_{\Sigma} \la\Db (X^\bot) , \Db (Z^\bot) \ra - Ric_\mathcal{N} (X^\bot, Z^\bot) - |A|^2\la X^\bot , Z^\bot \ra \id \h^n \nonumber \\
& + \int_{\Sigma} \la X^{\bot}, H\ra\la Z^{\bot},H \ra - \la \nabla_{X^{\bot}} (Z^{\bot}), H \ra - \la [X^{\bot},Z^{\top}],H \ra \id \h^n \nonumber \\
& + \int_{\de \Sigma} \la \D_X Z, \nu\ra + \omega(\nu) + div_{\Sigma}(Z^{\top})\la X^{\top}, \nu \ra - \la \nabla_{X^{\top}}(Z^{\top}),\nu \ra \id \h^{n-1}.
\end{align}  
where $\omega$ is is the one-form defined in (\ref{AeqO}). Since 
\begin{eqnarray} \label{Aeqbdry2}
\la \nabla_{X}Z , \nu \ra + \omega(\nu) & = & \la \nabla_{X^{\bot}}(Z^{\bot}),\nu \ra + \la \nabla_{X^{\bot}}(Z^{\top}),\nu \ra + \la \nabla_{X^{\top}}(Z^{\bot}),\nu \ra + \la \nabla_{X^{\top}}(Z^{\top}),\nu \ra \nonumber \\
& & + \la A(Z^{\top},\nu),X^{\bot}\ra + \la A(X^{\top},\nu),Z^{\bot}\ra - \la Z^{\bot}, H \ra \la X^{\top}, \nu \ra - \la X^{\bot}, H \ra \la Z^{\top}, \nu \ra \nonumber \\
& = & \la \nabla_{X^{\bot}}(Z^{\bot}),\nu \ra + \la [X^{\bot},Z^{\top}],\nu \ra + \la \nabla_{X^{\top}}(Z^{\top}),\nu \ra \nonumber \\
&   & - \la Z^{\bot}, H \ra \la X^{\top}, \nu \ra - \la X^{\bot}, H \ra \la Z^{\top}, \nu \ra,
\end{eqnarray}
combining (\ref{AeqX}) and (\ref{Aeqbdry2}) we obtain:
\begin{align*}
\pl{}{t}\left.\pl{}{s}\h^n(\Sigma(t,s))\right|_{t=s=0}  
=&\int_{\Sigma} \la\Db (X^\bot) , \Db (Z^\bot) \ra - Ric_\mathcal{N} (X^\bot, Z^\bot) - |A|^2\la X^\bot , Z^\bot \ra \id \h^n \\
& + \int_{\Sigma} \la X^{\bot}, H\ra\la Z^{\bot},H \ra - \la \nabla_{X^{\bot}} (Z^{\bot}), H \ra - \la [X^{\bot},Z^{\top}],H \ra \id \h^n \\
& + \int_{\de \Sigma} \la \D_{X^{\bot}} (Z^{\bot}), \nu\ra + \la [X^{\bot},Z^{\top}],\nu \ra + div_{\Sigma}(Z^{\top})\la X^{\top}, \nu \ra  \id \h^{n-1} \\
& - \int_{\de \Sigma}  \la Z^{\bot}, H \ra \la X^{\top}, \nu \ra + \la X^{\bot}, H \ra \la Z^{\top}, \nu \ra \id \h^{n-1}.
\end{align*} 
We remark that, in the above formula, whereas $X$ and $Z$ are tangent to $\partial \mathcal{N}$ by assumption, $X^{\top}, X^{\bot}, Z^{\top}$ and $Z^{\bot}$ are tangent to $\partial \mathcal{N}$ if and only if $\nu$ is orthogonal to $\partial \mathcal{N}$.

When $\Sigma$ is a free boundary minimal hypersurface in $\mathcal{N}$ we recover the usual second variation formula,
\begin{equation*}
\pl{}{t}\left.\pl{}{s}\h^n(\Sigma(t,s))\right|_{t=s=0} = \check{Q}_{\Sigma}(X^{\bot},Z^{\bot})=\check{Q}_{\Sigma}(Z^{\bot},X^{\bot}).
\end{equation*}
where $\check{Q}_{\Sigma}$ is the index form on the normal bundle of $\Sigma$. In fact, in that situation the terms containing $H$ are obviously zero and, since \textit{all} components $X^{\top}$, $X^{\bot}$, $Z^{\top}$, $Z^{\bot}$ are orthogonal to $\nu$ (i.e. they are tangent to $\partial \mathcal{N}$), $\la X^{\bot},\nu \ra$ vanishes identically and 
$$ \la [X^{\top},Z^{\bot}] ,\nu \ra = \la \nabla_{X^{\top}}(Z^{\bot}),\nu \ra - \la \nabla_{Z^{\bot}}(X^{\top}),\nu \ra = \la II^{\partial \mathcal{N}}(X^{\top},Z^{\bot}),\nu \ra - \la II^{\partial \mathcal{N}}(Z^{\bot},X^{\top}),\nu \ra = 0.$$

For the sake of the argument in Section \ref{sec:thm4}, it is convenient to compare $\nu$ with the unit outward pointing normal vector field to $\partial \mathcal{N}$, denoted by $\nu_{e}$. In a succinct way, we write

\begin{align*}
\pl{}{t}\left.\pl{}{s}\h^n(\Sigma(t,s))\right|_{t=s=0}  
=&\int_{\Sigma} \la\Db (X^\bot) , \Db (Z^\bot) \ra - Ric_\mathcal{N} (X^\bot, Z^\bot) - |A|^2\la X^\bot , Z^\bot \ra \id \h^n \\
& + \int_{\de \Sigma} \la \nabla_{X^{\perp}} Z^{\perp}, \nu_e \ra \id \h^{n-1} \\ 
& + \int_{\Sigma} \Xi_1(X,Z,H) \id \h^n + \int_{\partial \Sigma} \Xi_2(X,Z,H,\nu,\nu_e) \id \h^{n-1},
\end{align*}
where
\begin{equation*}
\Xi_1(X,Z,H) = \la X^{\bot}, H\ra\la Z^{\bot},H \ra - \la \nabla_{X^{\bot}} (Z^{\bot}), H \ra - \la [X^{\bot},Z^{\top}],H \ra
\end{equation*}
and
\begin{align*}
\Xi_2(X,Z,H,\nu,\nu_e) = & \la \nabla_{X^{\bot}}Z^{\bot},\nu-\nu_e \ra + \la [X^{\bot},Z^{\top}],\nu \ra +  div_{\Sigma}(Z^{\top})\la X^{\top}, \nu \ra \\
& - \la Z^{\bot}, H \ra \la X^{\top}, \nu \ra - \la X^{\bot}, H \ra \la Z^{\top},\nu \ra.
\end{align*}
$\Xi_1$ is linear in $X$ and $Z$, is of first order in these entries, and $\Xi_1 \to 0$ smoothly as $H\to 0$ smoothly. $\Xi_2$ depends on $X$ and $Z$ similarly, and $\Xi_2 \to 0$ smoothly as $H\to 0$ and $\nu\to \nu_e$ (for that forces $X^{\bot}$, $X^{\top}$, $Z^{\bot}$ and $Z^{\top}$ to become tangent to $\partial \mathcal{N}$)). In the case that $\Sigma$ is uniformly and smoothly close to a fixed minimal and free boundary hypersurface, we will have $\Xi_1, \Xi_2$ uniformly small, and $\la \nabla_{X^{\perp}} Z^{\perp}, \nu_e \ra$ uniformly close to $\la II^{\partial \mathcal{N}}( (X^{\bot})^{\top_{\partial \mathcal{N}}}, (Z^{\bot})^{\top_{\partial \mathcal{N}}}),\nu_e\ra$ (in a smooth sense).

\section{Proof of Proposition \ref{pro:EuclFol}}\label{app_fol}

\begin{proof}[Proof of Proposition \ref{pro:EuclFol}] 

Here we need to show that if $g$ is close enough to the Euclidean metric, and $t, w$ are also small enough, then one can find an implicitly defined function $u=u(t,g,w)$ so that $\Phi(t,g,w,u(t,g,w))=(0,0,0)$. 
                
                The functional $\Phi$ is $C^1$ (see e. g. Appendix of \cite{Whi87A}) and its differential can be computed using the result of Proposition 17 in \cite{Amb15}, so that
                \[
                D_4\Phi(t,\delta,0,0)[v]=\frac{d}{ds}_{|{s=0}}\Phi(t,\delta,0,sv)=\left(\Delta_{\delta}v,\frac{\partial v}{\partial \nu_{\delta}},v|_{\Gamma_2}\right)
                \]        
                where $\delta$ stands for the Euclidean metric and $\nu_{\delta}=-\frac{\partial}{\partial x^1}$ in the standard Euclidean coordinates we are adopting.
                
               In order to apply the implicit function theorem we need to study the mapping properties of the linear operator 
               \[
               D_4\Phi(t,\delta,0,0): \ Y\to Z_1\times Z_2\times Z_3.
               \]
               First of all, it follows at once from Theorem 6,I in \cite{Mir55} together with standard Schauder estimates that for any triple $(f_1,f_2,f_3)\in Z_1\times Z_2\times Z_3$ the problem
               \begin{equation}\label{eq:linsys}
               \begin{cases}
               \Delta_{\delta}v=f_1 & \textrm{in} \ S_{\theta}\\
               \frac{\partial v}{\partial \nu_{\delta}}=f_2 & \textrm{in} \ \Gamma_1 \\
               v=f_3 & \textrm{in} \ \Gamma_2
                  \end{cases}
               \end{equation}
               admits one (and only one) solution $u\in C(\overline{S_{\theta}})\cap C^{2,\alpha}(\Omega)$ where $\Omega$ is any relatively compact domain in $\overline{S_{\theta}}\setminus (\Gamma_1\cap \Gamma_2)$. Of course, this implies at once that the map above is injective. For what concerns surjectivity, it is then enough to invoke Theorem 1 in \cite{AK82}, which ensures that any \textsl{bounded} solution to \eqref{eq:linsys} does in fact belong to $C^{2,\alpha}(\overline{S_{\theta}})$ provided $\alpha<\alpha_0$ for $\alpha_0=\left\{\frac{\pi}{2\theta}\right\}$ where $\left\{x\right\}\in [0,1)$ is defined by $\left\{x\right\}=x-\floor{x}$. In fact, it follows from the explicit barrier construction presented in the first part of the proof of Lemma 1 in the same reference, aimed at handling the situation locally around the edge points, that one gains a global Schauder estimate of the form
               \begin{equation}\label{eq:Sch1}
               \|u\|_Y\leq C\left(\|u\|_{C^0(\overline{S_{\theta}})}+\|f_1\|_{Z_1}+\|f_2\|_{Z_2}+\|f_3\|_{Z_3}\right)
               \end{equation}
               for any solution $u\in Y$ of \eqref{eq:linsys}.
            In this respect, all we need to check is that this can in fact be upgraded to 		
            	\begin{equation*}\label{eq:Sch2}
            	\|u\|_Y\leq C'\left(\|f_1\|_{Z_1}+\|f_2\|_{Z_2}+\|f_3\|_{Z_3}\right)
            	\end{equation*}		
            			as this patently completes the proof that $D_4\Phi(t,\delta,0,0): \ Y\to Z_1\times Z_2\times Z_3$ is indeed a Banach isomorphism. This is rather standard: assuming by contradiction that were not the case, one could find a sequence $\left\{u_k\right\}\subset Y$ such that, possibly by renormalizing
            			\[
            			\begin{cases}
            			\|u_k\|_Y=1 \\
            			\|\Delta_{\delta}u_k\|_{Z_1}+\|\frac{\partial u_k}{\partial\nu_{\delta}}\|_{Z_2}+\|u_k\|_{Z_3}\leq 1/k
            			\end{cases}
            			\]
            			for all $k\geq 1$ and hence, invoking the Arzel\'a-Ascoli compactness theorem
            			\[
            				u_k\to u    \ \textrm{in} \ C^{2}(\overline{S_{\theta}})
            			\]
            			for a subsequence which we shall not rename. In particular, one has at the same time
            			\[
            			\Delta_{\delta}u_k \to \Delta_{\delta}u \  \ \textrm{in} \ C^{0}(\overline{S_{\theta}}), \ \textrm{and} \ \Delta_{\delta}u_k\to 0 \ \textrm{in} \ C^{0,\alpha}(\overline{S_{\theta}})
            				\]
            				hence $u\in C^{2}(\overline{S_{\theta}})$ must be a solution of the homogeneous problem
            				\begin{equation}\label{eq:linsysH}
            				\begin{cases}
            				\Delta_{\delta}u=0 & \textrm{in} \ S_{\theta}\\
            				\frac{\partial u}{\partial \nu_{\delta}}=0 & \textrm{in} \ \Gamma_1 \\
            				u=0 & \textrm{in} \ \Gamma_2
            				\end{cases}
            				\end{equation}
            		so that we conclude $u=0$ by virtue of the aforementioned result by C. Miranda. But then $u_k\to 0$ in $C^2(\overline{S_{\theta}})$ and thus inequality \eqref{eq:Sch1} applied to $u_k$ immediately implies that $\|u_k\|_Y\to 0$ as we let $k\to\infty$, contrary to the fact that each function of the sequence has been rescaled so to have norm one. This contradiction completes the proof.
            	\end{proof}


\begin{thebibliography}{99}
\bibitem{A16}
N.~S. Aiex, \textit{Non-compactness of the space of minimal hypersurfaces}, preprint (arXiv:1601.01049). 




\bibitem{Amb15}
L. Ambrozio, \textit{Rigidity of area-minimizing free boundary surfaces in mean convex three-manifolds}, J. Geom. Anal. \textbf{25} (2015), no. 2, 1001-1017.

\bibitem{ACS15}
L. Ambrozio, A. Carlotto and B. Sharp, \textit{Compactness of the Space of Minimal Hypersurfaces with Bounded Volume and $p$-th Jacobi Eigenvalue}, J. Geom. Anal. \textbf{26} (2016), no. 4, pages 2591-2601.

\bibitem{ACS16}
L. Ambrozio, A. Carlotto and B. Sharp, \textit{Index estimates for free boundary minimal hypersurfaces}, Math. Ann. (\textsl{to appear}).

\bibitem{AN16}
L. Ambrozio, I. Nunes, \textit{A gap theorem for free boundary minimal surfaces in the three-ball}, preprint (arXiv: 1608.05689).


\bibitem{AK82}
A. Azzam, E. Kreyszig, \textit{On solutions of elliptic equations satisfying mixed boundary conditions}, SIAM J. Math. Anal. \textbf{13} (1982), no. 2, 254-262. 



\bibitem{Cai35} S. Cairns, \textit{Triangulation of the manifold of class one}, Bull. Amer. Math. Soc. \textbf{41} (1935), no. 8, 549-552.			





\bibitem{CS85}
H. I. Choi and R. Schoen, \textit{The space of minimal embeddings of a surface into a three-dimensional manifold of positive Ricci curvature}, Invent. Math. \textbf{81} (1985), 387--394.


%
%

\bibitem{Cou40} R. Courant, \textit{The existence of minimal surfaces of given topological structure under prescribed boundary conditions}, Acta Math. \textbf{72} (1940), 51-98.

\bibitem{Cou50} R. Courant, \textit{Dirichlet's Principle, Conformal Mapping, and Minimal Surfaces. Appendix by M. Schiffer}, Interscience Publishers, Inc., New York, N.Y., 1950. xiii+330 pp.

\bibitem{DR16}
C. De Lellis and J. Ramic, \textit{Min-max theory for minimal hypersurfaces with boundary}, preprint (arXiv:1611.00926).

\bibitem{Dev16}
B. Devyver, \textit{Index   of   the   critical   catenoid}, preprint (arXiv: 1609.02315).

\bibitem{FPZ15} A. Folha, F. Pacard and T. Zolotareva, \textit{Free boundary minimal surfaces in the unit $3$-ball}, preprint (arXiv:1502.06812).
%
%
%
%

\bibitem{FL14}
A. Fraser and M. Li, \textit{Compactness of the space of embedded minimal surfaces with free boundary in three-manifolds with nonnegative Ricci curvature and convex boundary}, J. Differential Geom. \textbf{96} (2014), no. 2 , 183-200.

\bibitem{FS11}
A. Fraser and R. Schoen,  \textit{The first Steklov eigenvalue, conformal geometry, and minimal surfaces}, Adv. Math. \textbf{226} (2011), no. 5, 4011-4030. 
 
\bibitem{FS13}
A. Fraser and R. Schoen,  \textit{Minimal surfaces and eigenvalue problems. Geometric analysis, mathematical relativity, and nonlinear partial differential equations}, 105–121, Contemp. Math. \textbf{599}, Amer. Math. Soc., Providence, RI, 2013. 

\bibitem{FS16}
A. Fraser and R. Schoen, \textit{Sharp eigenvalue bounds and minimal surfaces in the ball}, Invent. Math. \textbf{203} (2016), no. 3, 823-890.
	
\bibitem{FGM16}
B. Freidin, M. Gulian and P. McGrath, \textit{Free boundary minimal surfaces in the unit ball with low cohomgeneity}, Proc. Amer. Math. Soc. \textbf{145}(4) (2017), 1671-1683. 

\bibitem{gt}
D.~Gilbarg and N.~S.~Trudinger,
\newblock {\em Elliptic partial differential equations of second order}.
\newblock Classics in Mathematics. Springer-Verlag, Berlin, 2001.


\bibitem{GJ86B}
 M. Gr\"uter, J. Jost,  \textit{Allard type regularity results for varifolds with free boundaries}, Ann. Scuola Norm. Sup. Pisa Cl. Sci. (4) \textbf{13} (1986), no. 1, 129-169. 




\bibitem{LZ16A}
Q. Guang, M. Li and X. Zhou, \textit{Curvature estimates for stable free boundary minimal hypersurfaces}, preprint (arXiv:1611.02605).
%
%
\bibitem{Hsi83}
W. Y. Hsiang, \textit{Minimal cones and the spherical Bernstein problem. I.}, Ann. of Math. \textbf{118} (1983), no. 1, 61-73. 

\bibitem{Ket16A}
D. Ketover, \textit{Free boundary minimal surfaces of unbounded genus}, preprint (arXiv:1612.08691).


\bibitem{Lan99}
S. Lang, \textit{Fundamentals of differential geometry}, Graduate Texts in Mathematics, 191. Springer-Verlag, New York, 1999.

\bibitem{Li15}
M. Li, \textit{A general existence theorem for embedded minimal surfaces with free boundary}, Comm. Pure Appl. Math. \textbf{68} (2015), no. 2, 286-331.

\bibitem{LZ16B}
M. Li and X. Zhou, \textit{Min-max theory for free boundary minimal hypersurfaces I - regularity theory}, preprint (arXiv:1611.02612).

\bibitem{MNS13}
D. M\'aximo, I. Nunes and G. Smith, \textit{Free boundary minimal annuli in convex three-manifolds}, J. Differential Geom. \textbf{106} (2017), no. 1, 139-186. 


\bibitem{M16} 
P. McGrath, \textit{A characterization of the critical catenoid}, preprint (arXiv:1603.04114v2).

\bibitem{Mir55}
C. Miranda, \textit{Sul problema misto per le equazioni lineari ellittiche}, Ann. Mat. Pura Appl. \textbf{39} (1955), 279-303.





%
\bibitem{SS81}
R. Schoen and L. Simon, \textit{Regularity of stable minimal hypersurfaces}, Comm. Pure Appl. Math. \textbf{34} (1981), 741--797.

\bibitem{SSY75}
R.~Schoen, L.~Simon and S.~T.~Yau, \textit{Curvature estimates for minimal hypersurfaces}, Acta Math. \textbf{134} (1975), 275--288.  

\bibitem{Sha15}
B. Sharp, \textit{Compactness of minimal hypersurfaces with bounded index}, J. Differential Geom. \textbf{106} (2017), no. 2, 317-339. 

\bibitem{Sim83}
L. Simon, \textit{Lectures on geometric measure theory}, Proceedings of the Centre for Mathematical Analysis, ANU \textbf{3} (1983). 


\bibitem{S68}
J.~Simons, \textit{Minimal Varieties in Riemannian Manifolds}, Ann. of Math. \textbf{88} (1968), 62--105.

\bibitem{Sma65} S. Smale, \textit{An infinite dimensional version of Sard's theorem}, Amer. J. Math. \textbf{87} (1965), 861-866. 

\bibitem{SZ16}
G. Smith, D. Zhou, \textit{The Morse index of the critical catenoid}, preprint (arXiv:1609.01485).



\bibitem{Tay81}
M. Taylor, \textit{Pseudodifferential operators}, Princeton Mathematical Series, 34. Princeton University Press, Princeton, N.J., 1981. xi+452 pp. 

\bibitem{Tra16}
H. Tran, \textit{Index  characterization  for  free  boundary  minimal  surfaces}, preprint (arXiv: 1609.01651).



\bibitem{Whi87A}
B. White, \textit{The space of m-dimensional surfaces that are stationary for a parametric elliptic functional}, Indiana Univ. Math. J. \textbf{36} (1987), no. 3, 567-602. 
            	   
\bibitem{Whi87B}
B. White, \textit{Curvature estimates and compactness theorems in 3-manifolds for surfaces that are stationary for parametric elliptic functionals}, Invent. Math. \textbf{88} (1987), no. 2, 243-256.
            	   

\bibitem{Whi91}
B. White, \textit{The space of minimal submanifolds for varying Riemannian metrics}, Indiana Univ. Math. J. \textbf{40} (1991), no. 1, 161-200. 



\bibitem{Whi05}
B.~White, \textit{A local regularity theorem
for mean curvature flow}, Ann. Math. \textbf{161} (2005), 1487-1519. 

\bibitem{Whi09}
B. White, \textit{Which ambient spaces admit isoperimetric inequalities for submanifolds?}, J. Diff. Geom. \textbf{83} (2009), 213-228.

\bibitem{Whi15} B.~White, \textit{On the Bumpy Metrics Theorem for Minimal Submanifolds}, Amer. J. Math. (\textsl{to appear}). 


\bibitem{Whi40}J. H. C. Whitehead, \textit{On $C^1$-complexes}, Ann. of Math. (2)
\textbf{41} (1940), 809-824.




\end{thebibliography}
\end{document}